\documentclass[11pt]{article}

\usepackage{amsthm,amssymb,amsfonts,amsmath,graphicx, color,mathrsfs,booktabs}
\usepackage{hyperref}
\usepackage{xcolor}
\usepackage{geometry}[margin=.9in]

\usepackage{algorithm}
\usepackage{algpseudocode}
\usepackage{booktabs}

\newcommand{\subopt}{{\mathsf{subopt}}}
\newcommand{\gap}{{\mathsf{gap}}}
\newcommand{\primaldual}{{\mathsf{primaldual}}}

\newcommand{\nuc}{\mathsf{nuc}}

\def\Rfwt{{R}}
\def\fwt{{S}}

\def\vv{{\mathbf{v}}}

\def\xx{{\mathbf{x}}}
\def\yy{{\mathbf{y}}}

\def\0{\mathbf{0}}
\def\1{\mathbf{1}}

\newcommand{\N}{\mathbb{N}}

\newcommand{\R}{\mathbb{R}}

\newcommand\cC{\mathcal{C}}

\newcommand\cI{{\ensuremath{\mathcal{I}}}}

\newcommand\cO{{\ensuremath{\mathcal{O}}}}

\DeclareMathOperator{\argmin}{arg\,min}

\def\fw{{\textnormal{\texttt{FW}}}}

\newcommand{\srb}[1]{\left( #1 \right)}

\theoremstyle{plain} \numberwithin{equation}{section}
\newtheorem{theorem}{Theorem}[section]
\numberwithin{theorem}{section}
\newtheorem{lemma}[theorem]{Lemma}

\newtheorem{remark}[theorem]{Remark}

\newcommand{\hrulealg}[0]{\vspace{1mm} \hrule \vspace{1mm}}

\usepackage{subcaption}
\begin{document}

\title{Adaptive Open-Loop Step-Sizes for Accelerated Convergence Rates of the Frank-Wolfe Algorithm}

\author{ Elias Wirth\thanks{Technical University of Berlin} \and Javier Pe\~na\thanks{Carnegie Mellon University} \and Sebastian Pokutta\thanks{Zuse Institute Berlin and Technical University of Berlin}}

\date{May 15, 2025}

\maketitle

\abstract{
Recent work has shown that in certain settings, the Frank-Wolfe algorithm (\fw{}) with open-loop step-sizes $\eta_t = \frac{\ell}{t+\ell}$ for a fixed parameter $\ell\in\N_{\ge 2}$, attains a convergence rate faster than the traditional $\cO(t^{-1})$ rate.  In particular, when a {\em strong growth property} holds, the convergence rate attainable with open-loop step-sizes $\eta_t = \frac{\ell}{t+\ell}$ is $\cO(t^{-\ell})$.  In this setting there is no single value of the parameter $\ell$ that prevails as superior.    

This paper shows that \fw{} with log-adaptive open-loop step-sizes $\eta_t = \frac{2+\log(t+1)}{t+2+\log(t+1)}$ attains a convergence rate that is at least as fast as that attainable with fixed-parameter open-loop step-sizes $\eta_t = \frac{\ell}{t+\ell}$ for any value of $\ell\in \N_{\ge 2}$.  To establish our main convergence results, we extend our previous affine-invariant accelerated convergence results for \fw{} to more general open-loop step-sizes $\eta_t = \frac{g(t)}{t+g(t)}$, where $g\colon\N\to\R_{\geq 0}$ is any non-decreasing function such that the sequence of step-sizes $\eta_t = \frac{g(t)}{t+g(t)}$ is non-increasing. This covers in particular the fixed-parameter case by choosing $g(t) = \ell$ and the log-adaptive case by choosing $g(t) = 2+ \log(t+1)$.

To facilitate the adoption of log-adaptive open-loop step-sizes, we have incorporated the log-adaptive step-size into the \texttt{FrankWolfe.jl} software package.
}

\section{Introduction}
We consider constrained convex optimization problems of the form
\begin{equation}\label{eq.opt}\tag{OPT}
    \min_{\xx\in\cC}f(\xx),
\end{equation}
where $\cC\subseteq\R^n$ is a compact convex set and $f\colon \cC \to \R$ is a convex and smooth function. When projection onto $\cC$ is hard,~\eqref{eq.opt} can be tackled with the \emph{Frank-Wolfe algorithm} (\fw) \cite{frank1956algorithm}, otherwise known as the \emph{conditional gradients algorithm} \cite{levitin1966constrained}. \fw{}, presented in Algorithm~\ref{alg:fw}, relies only on a \emph{linear minimization oracle} (LMO) to access the feasible region $\cC$, which is often much cheaper than a projection oracle onto $\cC$~\cite{combettes2021complexity}. During each iteration, \fw{} computes the vertex $\vv^{(t)}$ best-aligned with the negative gradient of the objective and takes a step in the direction $\vv^{(t)} - \xx^{(t)}$ of length $\eta_t$, where $\eta_t\in [0, 1]$ is determined according to a step-size rule. Popular step-size rules include
line-search $\eta_t \in \argmin_{\eta \in [0, 1]} f(\xx^{(t)} + \eta (\vv^{(t)} - \xx^{(t)}))$,
short-step $\eta_t = \min\left\{1,\frac{\langle \nabla f(\xx^{(t)}), \xx^{(t)} -\vv^{(t)} \rangle}{L\|\xx^{(t)} -\vv^{(t)}\|_2^2}\right\}$, adaptive \cite{pedregosa2018step}, and fixed-$\ell$ open-loop $\eta_t = \frac{\ell}{t+\ell}$ for $\ell\in\N_{\geq 1}$ \cite{dunn1978conditional,wirth2023acceleration,wirth2023accelerated}.

Under mild assumptions, \fw{} with any of the above choices of step-size $\eta_t$ achieves a convergence rate of $f(\xx^{(t)}) - \min_{\xx\in\cC} f(\xx) = \cO(t^{-1})$ \cite{jaggi2013revisiting}. Furthermore, \fw{} admits several attractive properties: The algorithm and its variants \cite{garber2016linear,holloway1974extension, lacoste2015global,  tsuji2022pairwise,wolfe1970convergence} are easy to implement, first-order, projection-free, affine-covariant\footnote{For details, see Definition~1.1 of \cite{wirth2023accelerated}.}, and their iterates are sparse convex combinations of vertices of the feasible region.
Therefore, \fw{} variants have been leveraged in a variety of settings: approximate vanishing ideal computations \cite{ wirth2023approximate,wirth2022conditional}, computer vision \cite{alayrac2016unsupervised, bojanowski2015weakly, joulin2014efficient,
    miech2017learning,
    peyre2017weakly,  seguin2016instance},
kernel herding \cite{bach2012equivalence, baskaran2022distribution, lacoste2015sequential, tsuji2022pairwise},
matrix completion \cite{allen2017linear, freund2017extended,  garber2018fast, mu2016scalable, rezaei2017background, shalev2011large},
submodular function optimization \cite{bach2013learning, bian2017guaranteed, mokhtari2018conditional}, tensor completion \cite{bugg2022nonnegative,guo2017efficient}, and the
training of support vector machines \cite{ clarkson2012sublinear, lacoste2013block, osokin2016minding,ouyang2010fast}.
See \cite{braun2022conditional} for a recent survey.

\begin{algorithm}[!t]
    \caption{Frank-Wolfe algorithm (\fw{})}\label{alg:fw}
	\begin{algorithmic}        
    \State{\bf Input:} {$\xx^{(0)}\in\cC$.
    }
    \State{\bf Output:} {$\xx^{(t)}\in\cC$ for $t\in\N$.}
    \hrulealg
    \For{$t\in\N$}
    \State {$\vv^{(t)} \gets \argmin_{\vv\in \cC} \langle \nabla f(\xx^{(t)}), \vv \rangle$ }
    \State {$\xx^{(t+1)} \gets \xx^{(t)} + \eta_t (\vv^{(t)}- \xx^{(t)})$
    for some $\eta_t \in [0,1]$}
    \EndFor 
\end{algorithmic}    
\end{algorithm}

In recent years, \fw{} with fixed-$\ell$ open-loop step-sizes of the form
\begin{align}\label{eq.open_loop_old}
    \eta_t = \frac{\ell}{t+\ell}
\end{align}
for some $\ell\in\N_{\geq 2}$ has been studied more closely \cite{carderera2021simple, dunn1978conditional, lan2013complexity,  li2021momentum}.
Formally speaking, an open-loop step-size is a sequence $(\eta_t)_{t\in\N}$ of step-sizes used in an iterative optimization algorithm, where each $\eta_t$ is determined independently of the optimization problem's parameters and the algorithm's current state or iterate values.
This is in stark contrast to \fw{} with line-search, which requires feedback from the objective function or short-step and adaptive step-sizes \cite{garber2015faster, ghadimi2019conditional, kerdreux2021projection, nesterov2018complexity}, which require prior knowledge of challenging-to-compute parameters of the optimization problem~\eqref{eq.opt}.  The recent articles by Wirth et al.~\cite{wirth2023acceleration, wirth2023accelerated} characterize several settings for which \fw{} with open-loop step-sizes admits convergence rate of order $\cO(t^{-2})$ or faster.  We summarize the main results in~\cite{wirth2023accelerated} and their relation to the developments in this paper in Table~\ref{table:results}.  In particular, when the strong $(M, 1)$-growth property holds, \fw{} with $\eta_t = \frac{\ell}{t+\ell}$ converges at a rate of order $\cO(t^{-\ell})$.
Thus, the long-considered `default' open-loop step-size
\begin{align*}
    \eta_t = \frac{2}{t+2}
\end{align*}
is not optimal. Indeed, no fixed-$\ell$ step-size of the form~\eqref{eq.open_loop_old} can be optimal, as increased $\ell$ leads to increased convergence speed in the strong $(M, 1)$-growth setting.

\begin{table}[t!]
    \centering
    \resizebox{\textwidth}{!}{
    \begin{tabular}{@{}l@{\quad}c@{\quad}l@{\quad}c@{\quad}l@{}} 
        \toprule
        Growth setting & $\eta_t = \frac{\ell}{t+\ell}$ & In \cite{wirth2023accelerated} & $\eta_t = \frac{2+\log(t+1)}{t+2+\log(t+1)}$ & This paper \\ 
        \midrule
        Strong $(M, 1)$ & $\cO(t^{-\ell})$ & Theorem~3.2 & $\tilde\cO(t^{-k})$ for any $k\in \N$ & Theorem~\ref{thm:strong_M_r_growth} \\ 
        Strong $(M, r)$ & $\cO(t^{-\ell+\epsilon}+t^{-\frac{1}{1-r}})$ & Theorem~3.4 & $\tilde\cO(t^{-\frac{1}{1-r}})$ & Theorem~\ref{thm:strong_M_r_growth} \\ 
        Weak $(M, r)$ & $\cO(t^{-\ell+\epsilon}+t^{-\frac{1}{1-r}} +t^{-2})$ & Theorem~4.1 & $\tilde\cO(t^{-\frac{1}{1-r}}+t^{-2})$ & Theorem~\ref{thm:rate-weak} \\ 
        \bottomrule
    \end{tabular}
    }
    \caption{Comparison of open-loop step-sizes for \fw{}. Growth settings are detailed in~\eqref{eq.strong_gp} and~\eqref{eq.weak_gp}. In the table, $r\in [0, 1[$, $\ell\in\N_{\geq 2}$, $\epsilon \in]0, \ell[$, $\fwt\in\N_{\geq 1}$ has to be chosen large enough according to conditions outlined in the corresponding theorems, and  $\tilde\cO(\cdot)$ ignores polylogarithmic factors. Rates always hold for the suboptimality gap $\subopt_t$. In the strong growth settings, results extend to the primal-dual suboptimality gap $\primaldual_t$.}
    \label{table:results}
\end{table}

The main goal of this paper is to revisit open-loop step-sizes and develop a modern optimal `default' open-loop step-size that performs at least as well as any open-loop step-size of the form \eqref{eq.open_loop_old} for any $\ell\in\N_{\geq 2}$ up to polylogarithmic factors. To that end, we move away from fixed-$\ell$ open-loop step-sizes and instead study \emph{adaptive open-loop step-sizes} of the form
\begin{align}\label{eq.new_olss}
    \eta_t = \frac{g(t)}{t+g(t)},
\end{align}
where $g\colon\N\to \R_{\geq 2}$. We are particularly interested in functions $g$ that gradually increase as $t$ increases, so as to achieve optimal convergence rates in the strong $(M,1)$-growth and strong $(M,r)$-growth settings.
Our contributions are as follows:
\begin{enumerate}
    \item \textbf{Convergence assumptions:} We extend the techniques introduced in~\cite{wirth2023accelerated} to develop affine-invariant convergence rates for \fw{} with open-loop step-sizes of the form \eqref{eq.new_olss}.  To that end, we identify the following two key assumptions about the function $g$:
    
    \begin{enumerate}
\item[A1.]  The function $g:\N\rightarrow \R_{\ge 2}$ is non-decreasing.
\item[A2.]  For all $t\in \N$ it holds that $\frac{t}{g(t)}  \le \frac{t+1}{g(t+1)}$ or equivalently  $\frac{g(t+1)}{t+1+ g(t+1)} \le \frac{g(t)}{t+g(t)}$.
\end{enumerate}

 These assumptions are crafted to ensure, and under suitable conditions accelerate, the convergence of \fw{} with adaptive open-loop step-sizes.

    \item \textbf{Log-adaptive open-loop step-size:} We advocate for a new standard in open-loop step-sizes: the \emph{log-adaptive open-loop step-size}, defined as
          \begin{align}\label{eq.default_olss}
              \eta_t = \frac{2 + \log(t+1)}{t+2 + \log(t+1)}
          \end{align}
          and first mentioned in a remark in \cite{pokutta2023frankwolfe}.
          This step-size, derived by setting $g(t) = 2 + \log(t+1)$ in \eqref{eq.new_olss}, not only matches but, in certain growth settings, surpasses the convergence rates of all fixed-$\ell$ step-sizes, up to polylogarithmic factors. Specifically, it offers arbitrarily fast sublinear rates in the strong $(M,1)$-growth setting, and more efficient rates in the strong $(M, r)$-growth setting compared to fixed-$\ell$ step-sizes. See Table~\ref{table:results} for an overview.
    \item \textbf{Software:} To facilitate widespread adoption, we integrate step-sizes of the form \eqref{eq.new_olss} into the \texttt{FrankWolfe.jl} package \cite{besanccon2022frankwolfe}. In addition, we report some numerical experiments that corroborate the theoretical findings.
\end{enumerate}

The main sections of the paper are organized as follows.  Section~\ref{sec.blueprint} contains a blueprint that underlies the main convergence results in~\cite{wirth2023accelerated} for fixed-parameter open-loop step-sizes and subsequently describe a modification of this blueprint that underlies the results for adaptive open-loop step-sizes developed in this paper.  Section~\ref{sec.strong_growth} and Section~\ref{sec.weak_growth} detail the proofs of our two main technical results, namely Theorem~\ref{thm:strong_M_r_growth} and Theorem~\ref{thm:rate-weak}. Section~\ref{sec.experiments} presents some numerical experiments that illustrate the performance of adaptive open-loop step-sizes versus fixed-parameter open-loop step-sizes.

\medskip

We conclude this introduction with a simple but important observation.  Some straightforward calculations show that both Assumptions A1 and A2 are satisfied for the following choices of $g:\N\rightarrow \R_{\ge 2}$:
\begin{itemize}
\item $g(t) = \ell$ for some $\ell \in \R_{\ge 2}$,
\item $g(t) = 2+\log(t+1)$.
\end{itemize}
The above two choices of $g$ yield, respectively, fixed-parameter open-loop step-sizes $\eta_t = \frac{\ell}{t+\ell}$ and log-adaptive open-loop step-sizes $\eta_t = \frac{2+\log(t+1)}{t+2+\log(t+1)}$.  The latter will play a particularly important role in the sequel.

\section{Framework and General Convergence Blueprint}\label{sec.blueprint}

\subsection{Preliminaries} 
Throughout, let $\cC\subseteq \R^n$ be a compact convex set and let $f\colon \cC \to \R$ be convex and differentiable in an open set containing $\cC$.
Recall that the \emph{Bregman divergence} $D_f:\cC\times\cC\rightarrow \R$  of $f$ is defined as
\begin{align*}
    D_f(\yy,\xx) = f(\yy) - f(\xx) - \left\langle\nabla f\left(\xx\right), \yy-\xx \right\rangle.
\end{align*}
The {\em Frank-Wolfe gap} function
$\gap:\cC\rightarrow \R$ is defined as follows.  For $\xx\in\cC$ let $\vv \in \argmin_{\yy\in \cC}\left\langle\nabla f\left(\xx\right), \yy-\xx \right\rangle$ and
\begin{align*}
    \gap(\xx):= \left\langle\nabla f\left(\xx\right), \xx - \vv\right\rangle.
\end{align*}
A straightforward algebraic manipulation shows that for $\xx\in\cC$ we have 
$$\vv \in \argmin_{\yy\in \cC}\left\langle\nabla f\left(\xx\right), \yy-\xx \right\rangle,$$ and $\eta \in [0,1]$ the following objective reduction identity holds: 
\begin{equation}\label{eq.gap}
f(\xx+\eta(\vv-\xx))  = f(\xx) -\eta\gap(\xx) + D_f(\xx+\eta(\vv-\xx),\xx)
\end{equation}
Recall the following strong and weak growth properties introduced in \cite{Pena23,wirth2023accelerated}.  Suppose $M > 0$ and $r\in [0,1]$.
    \begin{enumerate}
        \item We say that the pair $(\cC,f)$ satisfies the
        \emph{strong $(M,r)$-growth property} if for all $\xx \in \cC$ and $\vv \in \argmin_{\yy\in \cC}\left\langle\nabla f\left(\xx\right), \yy-\xx \right\rangle$, it holds that
              \begin{align}\label{eq.strong_gp}
                  D_f(\xx+\eta(\vv-\xx),\xx) \leq \frac{M \eta^2}{2} \gap(\xx)^r \qquad \text{for all} \ \eta \in [0, 1].
              \end{align}
        \item We say that the pair $(\cC,f)$ satisfies the {\em weak $(M,r)$-growth property} if for all $\xx \in \cC$ and $\vv \in \argmin_{\yy\in \cC}\left\langle\nabla f\left(\xx\right), \yy-\xx \right\rangle$, it holds that
              \begin{align}\label{eq.weak_gp}
                  D_f(\xx+\eta(\vv-\xx),\xx)\cdot {\subopt(\xx)}^{1-r} \leq \frac{M \eta^2}{2} \gap(\xx) \qquad \text{for all} \ \eta \in [0, 1].
              \end{align}
    \end{enumerate}

The papers~\cite{Pena23,wirth2023accelerated} describe several popular classes of problem instances  $(\cC,f)$ that satisfy the above growth properties.

\subsection{Blueprint for convergence results}

The main convergence results for fixed-parameter open-loop step-sizes $\eta_t = \frac{\ell}{t+\ell}$  with $\ell \in \N_{\ge2}$ in~\cite{wirth2023accelerated}, namely~\cite[Theorem 3.2, Theorem 3.4, Theorem 4.1]{wirth2023accelerated}, follow by putting together the following three ingredients:
\begin{enumerate}
\item Cumulative  product bound: for $S\in \N_{\ge 1},\; \epsilon \in ]0,\ell[,$ and $t\in \N_{\ge S}$,
\[
\prod_{i=S}^t \left(1-\left(1-\frac{\epsilon}{\ell}\right) \eta_i\right)  \leq 
\left(\frac{\eta_t}{\eta_{S-1}} \right)^{\ell - \epsilon} \exp\left(\frac{\epsilon \ell}{S}\right).
\]

\item Objective reduction identity~\eqref{eq.gap}.

\item Strong $(M,r)$ growth~\eqref{eq.strong_gp} or weak $(M,r)$ growth~\eqref{eq.weak_gp} for some $M>0$ and $r\in [0,1[$.

\end{enumerate}

Wirth et al.~\cite{wirth2023accelerated} showed that the combination of the above three ingredients yields the following accelerated convergence rates for \fw{} with step-sizes  $\eta_t = \frac{\ell}{t+\ell}:$ 
$\cO(t^{-\ell+\epsilon} + t^{-\frac{1}{1-r}})$ when the strong growth property holds~\cite[Theorem 3.4]{wirth2023accelerated}
and $\cO(t^{-2} + t^{-\frac{1}{1-r}})$ when the weak growth property holds~\cite[Theorem 4.1]{wirth2023accelerated}.

\medskip

In this paper we extend the above blueprint by modifying the first ingredient: For $\eta_t = \frac{g(t)}{t+g(t)}$ where $g:\N \rightarrow \R_{\ge 1}$ satisfies Assumptions A1 and A2, the first ingredient above can be extended as follows.
\begin{itemize}
\item New cumulative  product bound: for $S\in \N_{\ge 1},\; \epsilon \in ]0,g(S)[,$ and $t\in \N_{\ge S}$,
\[
\prod_{i=S}^t \left(1-\left(1-\frac{\epsilon}{g(t)}\right) \eta_i\right)  \leq 
\left(\frac{\eta_t}{\eta_{S-1}} \right)^{g(S) - \epsilon}. 
\]
\end{itemize}

For $g(t) = 2+\log(t+1)$ the new cumulative  product bound allows us to tweak our previous proof techniques to show the convergence rate $\tilde \cO(t^{-\frac{1}{1-r}})$ when the strong growth property holds (Theorem~\ref{thm:strong_M_r_growth}) and $\tilde \cO(t^{-2} + t^{-\frac{1}{1-r}})$ when the weak growth property holds (Theorem~\ref{thm:rate-weak}) .  In this notation $\tilde \cO(\cdot)$ hides some polylogarithmic factors.

Our convergence results  give bounds on the following two optimality measures. The first one is the
{\em primal suboptimality gap} function $\subopt\colon\cC\rightarrow \R$, defined as
\begin{align*}
    \subopt(\xx):= f(\xx) - \min_{\yy \in \cC} f(\yy).
\end{align*}
The second one is the \emph{best primal-dual suboptimality gap} defined as follows.  For a sequence $\xx^{(t)}\in\cC, \; t = 0,1,\dots$ define
\begin{align*}
    \primaldual(\xx^{(t)}) = \min_{k\in\{0,1,\ldots,t\}}\left\{f(\xx^{(t)}) - f(\xx^{(k)}) + \gap(\xx^{(k)})\right\}.
\end{align*}
We will write
$\gap_t$, $\primaldual_t$, and $\subopt_t$
as shorthands for $\gap(\xx^{(t)})$, $\primaldual(\xx^{(t)})$, and $\subopt(\xx^{(t)})$, respectively, when $\xx^{(t)}$ is the $t$-th iterate generated by Algorithm~\ref{alg:fw}.  In this case, as detailed in~\cite{wirth2023accelerated}, it holds that
\begin{align*}
    \subopt_t \leq \primaldual_t \leq \gap_t.
\end{align*}
Furthermore, identity~\eqref{eq.gap} and some straightforward calculations show that the \fw{} iterates also satisfy
\begin{align}\label{eq.fw-step}
    \subopt_{t+1} = \subopt_t -\eta_t \gap_t + D_f(\xx^{(t)}+\eta_t(\vv^{(t)}-\xx^{(t)}),\xx^{(t)}),
\end{align}
and
\begin{align}\label{eq.pridualgap}
    \primaldual_{t+1} \leq \primaldual_t-\eta_t \gap(\xx^{(t)}) + D_f(\xx^{(t)}+\eta_t(\vv^{(t)}-\xx^{(t)}),\xx^{(t)}).
\end{align}

\section{Strong growth setting}\label{sec.strong_growth}

We present the main technical lemma of the manuscript.

\begin{lemma}[Strengthened cumulative-product bound]\label{eq.main_technical.prod_sum_advanced}
Let $g: \N \to \R_{\geq 2}$ be non-decreasing and let $\eta_t:=\frac{g(t)}{t+g(t)}$ for $t\in \N$.  For integers $1 \leq S \leq t$ and any
$\epsilon\in]0, g(S)[$ the following holds
\begin{equation}\label{eq:cpb-short}
    \prod_{i=S}^{t}
    \left(1-\left(1-\frac{\epsilon}{g(i)}\right)\eta_i\right)
    \leq 
    \left(\frac{\eta_t}{\eta_{S-1}}\right)^{g(S)-\epsilon}.
\end{equation}
\begin{proof}
For ease of notation let $k:=g(S)$.
Since $g(i)\ge g(S) = k$ for all $i\ge S$, each factor satisfies
\[
    1-\left(1-\frac{\epsilon}{g(i)}\right)\eta_i
    = \frac{i+\epsilon}{i+g(i)}
    \leq \frac{i+\epsilon}{i+k)} = 1 - \frac{k-\epsilon}{i+k}.
\]
Hence if we set 
\[
    P:=\prod_{i=S}^t\left(1-\left(1-\frac{\epsilon}{g(i)}\right)\eta_i\right),
\]
the inequality $\log(1-x)\leq -x$ for $x<1$ implies that
\[
\log P
    \leq
    \sum_{i=S}^t
 \log\left(1 - \frac{k-\epsilon}{i+k}\right)
    \leq
    -\sum_{i=S}^{t}\frac{k-\epsilon}{i+k}.
\]
Thus
\begin{equation}
\label{eq:log_P.again}    
\log P
    \leq
    -(k-\epsilon)\sum_{i=S}^t\frac{1}{i+k}
    = -(k-\epsilon)\sum_{j=S+k}^{t+k}\frac{1}{j}
    \leq
    -(k-\epsilon)\log\frac{t+k}{S-1+k}.
\end{equation}
Since $g$ is non-decreasing and $\eta_i =\frac{g(i)}{i+g(i)}$ for $i\in \N$ it follows that 
\[
\frac{S-1+k}{t+k} = \frac{S-1+g(S)}{t+g(S)} \le  \frac{g(t)(S-1+g(S-1))}{(t+g(t))g(S-1)} =\frac{\eta_t}{\eta_{S-1}}
\] 
because $g(S-1) \le g(S) \le g(t)$ implies that 
\begin{align*}
&g(S-1)((S-1)t+(S-1)g(t)+g(S)t) \le g(t)((S-1)t+(S-1)g(S)+g(S-1)t) \\&\Leftrightarrow 
(S-1+g(S))(t+g(t))g(S-1)\le  g(t)(S-1+g(S-1))(t+g(S))\\&\Leftrightarrow 
\frac{S-1+g(S)}{t+g(S)} \le  \frac{g(t)(S-1+g(S-1))}{(t+g(t))g(S-1)}.
\end{align*}
Thus exponentiating~\eqref{eq:log_P.again} gives
\[
    P \leq
    \exp\left(-(k-\epsilon)\log\frac{t+k}{S-1+k}\right)
    \leq
    \exp\left((k-\epsilon)\log\frac{\eta_t}{\eta_{S-1}}\right)
    = \left(\frac{\eta_t}{\eta_{S-1}}\right)^{k-\epsilon},
\]
which is exactly \eqref{eq:cpb-short}.
\end{proof}
\end{lemma}

Among the settings we cover, the strong growth setting permits the most acceleration of \fw{}.

\begin{theorem}[Strong $(M,r)$-growth]\label{thm:strong_M_r_growth}
    Let $\cC\subseteq \R^n$ be a compact convex set, let $f\colon\cC\to\R$ be convex and differentiable in an open set containing $\cC$, let $g\colon\N\to\R_{\geq 2}$ satisfy Assumptions~A1~and~A2, and suppose that $(\cC, f)$ satisfies the strong $(M,r)$-growth property for some $M > 0$ and $r\in[0, 1[$. Let $\fwt\in\N_{\geq 1}$, $\epsilon \in ]0, g(\fwt) [$, and $k := \min\{g(\fwt)-\epsilon, \frac{1}{1-r}\}$.
    Then, for the iterates of Algorithm~\ref{alg:fw} with step-size $\eta_t = \frac{g(t)}{t+g(t)}$ and all $t\in\N_{\geq \fwt}$, it holds that
    \begin{align}\label{eq.rate_strong}
        \primaldual_{t} &\leq   \max\left\{ \primaldual_\fwt \left(\frac{\eta_{t-1}}{\eta_{\fwt-1}}\right)^{ g(\fwt)-\epsilon}, \left(\frac{M g(t-1)}{2\epsilon}\right)^{\frac{1}{1-r}}  \eta_{t-1}^{k}\right\}\notag \\& = \cO(g(t-1)^{\frac{1}{1-r}} \eta_{t-1}^{k}).
    \end{align}
    
In particular, for $g(t) = 2+\log(t+1), \; \epsilon = 1,\; S = \lceil\exp(\frac{1}{1-r})\rceil,$ and $t\in \N_{\ge S}$ it holds that
\begin{equation}\label{eq.adapt.strong}
\primaldual_{t} = \cO(\log(t)^{\frac{1}{1-r}} \eta_{t-1}^{\frac{1}{1-r}}) = \tilde \cO(t^{-\frac{1}{1-r}}). 
\end{equation}
\end{theorem}
\begin{proof}
    The proof is by induction and an extension of the proof of \cite[Theorem~3.4]{wirth2023accelerated}. First,~\eqref{eq.rate_strong} readily holds for $t  = \fwt$. For the main inductive step, suppose that~\eqref{eq.rate_strong} holds for $t\in\N_{\geq \fwt}$.
    We proceed by considering two possible cases:
    \begin{enumerate}
        \item \label{case:good_progress} Suppose that
              \begin{align}\label{eq.strong_r_gap}
                  \gap_t \leq \left(\frac{M  g(t) \eta_t }{2\epsilon}\right)^{\frac{1}{1-r}}.
              \end{align}
              By~\eqref{eq.pridualgap}, the strong $(M, r)$-strong growth property, and $\primaldual_t\leq\gap_t $, it holds that
              \begin{align}
                  \primaldual_{t+1} & \leq (1-\eta_t) \gap_t + \frac{M\eta_t^2}{2}\gap_t^r\nonumber                                                                                                                                                                                                      \\
                                    & \leq (1 - \eta_t) \left(\frac{M g(t) \eta_t }{2\epsilon}\right)^{\frac{1}{1-r}} + \eta_t \cdot \frac{M\eta_t}{2}\left(\frac{M g(t) \eta_t}{2\epsilon}\right)^{\frac{r}{1-r}}\nonumber                                                                              \\
                                    & \leq (1 - \eta_t) \left(\frac{M g(t) \eta_t }{2\epsilon}\right)^{\frac{1}{1-r}} + \eta_t \cdot \frac{M g(t) \eta_t}{2\epsilon}\left(\frac{M g(t) \eta_t}{2\epsilon}\right)^{\frac{r}{1-r}} & \text{$\triangleright$ because $\epsilon < g(S) \le g(t)$}\nonumber \\
                                    & = \left(\frac{M g(t) \eta_t }{2\epsilon}\right)^{\frac{1}{1-r}} \label{eq.good_growth}                                                                                                                                                                             \\
                                    & \leq \left(\frac{M g(t)}{2\epsilon}\right)^{\frac{1}{1-r}}  \eta_{t}^{k} , \nonumber 
              \end{align}
              where $k = \min\{ g(\fwt) - \epsilon, \frac{1}{1-r}\}$.
        \item\label{case:bad_progress}  Suppose that~\eqref{eq.strong_r_gap} does not hold. Let $\Rfwt\in\{\fwt, \fwt+1, \ldots, t\}$ be the smallest index such that
              \begin{align*}
                  \gap_i & \geq \left(\frac{M  g(i) \eta_i}{2\epsilon}\right)^{\frac{1}{1-r}}
              \end{align*}
              for all $i\in\{\Rfwt, \Rfwt+1,\ldots,t\}$. Then, Inequality~\eqref{eq.gap}, the strong $(M, r)$-growth property, and $\gap_i\geq \primaldual_i$ for all $i\in\N$ imply that
              \begin{align*}
                  \primaldual_{i+1} & \leq \primaldual_{i}+\eta_{i}\gap_{i}\srb{\frac{M \eta_{i}}{2\gap_{i}^{1-r}}-1} \leq  \primaldual_{i} \left(1 -\srb{1-\frac{\epsilon}{g(i)}} \eta_{i}\right)
              \end{align*}
              for all $i\in\{\Rfwt, \Rfwt +1, \ldots, t\}$. Thus,
              \begin{align}\label{eq.calc_big}
                  \primaldual_{t+1} & \leq \primaldual_{\Rfwt}\prod_{i=\Rfwt}^{t} \left(1 - \srb{1-\frac{\epsilon}{g(i)}}\eta_i\right)\nonumber                                      \\
                    & \leq 
                      \primaldual_\Rfwt \left(\frac{\eta_t}{\eta_{\Rfwt-1}}\right)^{ g(\Rfwt)-\epsilon}. & \text{$\triangleright$ by~\eqref{eq.main_technical.prod_sum_advanced}}
              \end{align}
              We distinguish between two subcases:
              \begin{enumerate}
                  \item\label{case:strong_r_a} $\Rfwt = \fwt$. In this case,~\eqref{eq.calc_big} implies that
                        \begin{align*}
                            \primaldual_{t+1} & \leq   \primaldual_\fwt \left(\frac{\eta_t}{\eta_{\fwt-1}}\right)^{ g(\fwt)-\epsilon}.
                        \end{align*}
                  \item \label{case:strong_r_b} $\Rfwt > \fwt$. In this case, $\gap_{\Rfwt-1}\leq \left(\frac{Mg(\Rfwt-1)\eta_{\Rfwt-1}}{2\epsilon}\right)^{\frac{1}{1-r}}$. We apply~\eqref{eq.good_growth} in Case~\ref{case:good_progress} to get $\primaldual_\Rfwt \leq \left(\frac{M g(\Rfwt-1)\eta_{\Rfwt-1}}{2\epsilon}\right)^{\frac{1}{1-r}}$. Thus,~\eqref{eq.calc_big} gives
                        \begin{align*}
                            \primaldual_{t+1} & \leq \left(\frac{M g(\Rfwt-1)\eta_{\Rfwt-1}}{2\epsilon}\right)^{\frac{1}{1-r}} \left(\frac{\eta_t}{\eta_{\Rfwt-1}}\right)^{ g(\Rfwt)-\epsilon}.
                        \end{align*}
                        Taking the worst possible $\Rfwt > \fwt$ yields
                        \begin{align*}
                            \primaldual_{t+1} & \leq \max_{\Rfwt\in\{\fwt,\fwt+1,\ldots,t\}}\left(\frac{M g(\Rfwt-1)\eta_{\Rfwt-1}}{2\epsilon}\right)^{\frac{1}{1-r}} \left(\frac{\eta_t}{\eta_{\Rfwt-1}}\right)^{ g(\Rfwt)-\epsilon}\nonumber                                                                                     \\
                                              & \leq \left(\frac{M g(t)}{2\epsilon}\right)^{\frac{1}{1-r}}  \max_{\Rfwt\in\{\fwt,\fwt+1,\ldots,t\}} \frac{\eta_{t}^{ g(\fwt) - \epsilon}}{\eta_{\Rfwt-1}^{ g(\fwt) - \epsilon- \frac{1}{1-r}}} & \text{$\triangleright$ by A1} \nonumber \\
                                              & \leq  \left(\frac{M g(t)}{2\epsilon}\right)^{\frac{1}{1-r}}  \eta_{t}^{k}.                                                                                                                                                                           \end{align*}
The last step follows because Assumption A2 implies that $0 < \eta_{i+1} \le \eta_i \le 1$ for all $i\in \N$ and thus
\begin{align*}
\max_{\Rfwt\in\{\fwt,\fwt+1,\ldots,t\}} \frac{\eta_{t}^{ g(\fwt) - \epsilon}}{\eta_{\Rfwt-1}^{ g(\fwt) - \epsilon- \frac{1}{1-r}}} &= 
\left\{\begin{array}{ll} \eta_{t}^{\frac{1}{1-r}}, & \text{ if }  g(\fwt) - \epsilon \geq \frac{1}{1-r} \\
            \frac{\eta_{t}^{ g(\fwt) - \epsilon}}{\eta_{S-1}^{ g(\fwt) - \epsilon- \frac{1}{1-r}}} ,     & \text{ if }  g(\fwt) - \epsilon \leq \frac{1}{1-r} 
        \end{array}\right.   \\                                                                                                                                         &\leq \eta_{t}^{\min\{ g(\fwt) - \epsilon, \frac{1}{1-r}\}}
\\ & = \eta_{t}^k.
\end{align*}

              \end{enumerate}
    \end{enumerate}
    Cases~\ref{case:good_progress},~\ref{case:strong_r_a}, and~\ref{case:strong_r_b} imply~\eqref{eq.rate_strong}.

\bigskip

Finally, for $g(t) = 2+\log(t+1), \; \epsilon = 1,\; S = \lceil\exp(\frac{1}{1-r})\rceil,$ and $t\in \N_{\ge S}$ we have
\[
S\ge \exp\left(\frac{1}{1-r}\right)-1 \Rightarrow g(S) = 2+\log(S+1) \ge \frac{1}{1-r} + 2 = \frac{1}{1-r} + 1 + \epsilon.
\]
Hence $ g(\fwt) -\epsilon \ge \frac{1}{1-r}$
and thus~\eqref{eq.adapt.strong} follows from~\eqref{eq.rate_strong} and the choice of $g(t) = 2+\log(t+1)$.

\end{proof}

\section{Weak growth setting}\label{sec.weak_growth}

In the weak growth setting, attainable acceleration is restricted compared to the strong $(M,r)$-growth setting in Theorem~\ref{thm:strong_M_r_growth}.

\begin{theorem}[Weak $(M, r)$-growth]\label{thm:rate-weak}
    Let $\cC\subseteq \R^n$ be a compact convex set, let $f\colon\cC\to\R$ be convex and differentiable in an open set containing $\cC$, let $g\colon\R_{\geq 0}\to\R_{\geq 0}$ satisfy Assumptions~A1~and~A2, and suppose that $(\cC, f)$  satisfies the strong $(M,0)$- and weak $(M, r)$-growth properties for some $M >0$ and $r\in [0, 1[$. Let $\fwt\in\N_{\geq 1}$, $\epsilon \in ]0,  g(\fwt)[$, and $k := \min\{ g(\fwt) -\epsilon , \frac{1}{1-r}, 2\}$.
    Then, for the iterates of Algorithm~\ref{alg:fw} with step-size $\eta_t = \frac{g(t)}{t+g(t)}$ and all $t\in\N_{\geq \fwt}$, it holds that
    \begin{align}\label{eq.rate_weak}
        \subopt_{t} & \leq
         \max\left\{ \subopt_{\fwt} \left(\frac{\eta_{t-1}}{\eta_{\fwt-1}}\right)^{ g(\fwt) - \epsilon}, \srb{\srb{\frac{M g(t-1)}{2\epsilon}}^{\frac{1}{1-r}} +\frac{M}{2}} \eta_{t-1}^{k}\right\} \notag \\ & = \cO\left(g(t-1)\eta_{t-1}^k\right).
    \end{align}
    
In particular, for $g(t) = 2+\log(t), \; \epsilon = 1,\; S = \lceil\exp(\frac{1}{1-r})\rceil,$ and $t\in \N_{\ge S}$ it holds that
\[
\subopt_{t} = \cO(\log(t)^{\frac{1}{1-r}} (\eta_{t-1}^{\min\{\frac{1}{1-r},2\}})) = \tilde \cO(t^{-\frac{1}{1-r}}+t^{-2}). 
\]

\end{theorem}
\begin{proof}
    The proof is by induction and an extension of the proof of \cite[Theorem~4.1]{wirth2023accelerated}. First,~\eqref{eq.rate_weak} readily holds for $t  = \fwt$. For the main inductive step, suppose that~\eqref{eq.rate_weak} holds for $t\in\N_{\geq \fwt}$.
    We distinguish between two main cases:
    \begin{enumerate}
        \item \label{case:good_progress.weak} Suppose that
              \begin{align}\label{eq.smaller.weak}
                  \subopt_{t} \leq \srb{\frac{M g(t) \eta_{t}}{2\epsilon}}^{\frac{1}{1-r}}.
              \end{align}
              Then, Equality~\eqref{eq.fw-step} and the strong $(M, 0)$-growth property imply that
              \begin{align}
                  \label{eq.good_growth.weak}
                  \subopt_{t+1} & \leq \subopt_{t} + \frac{M\eta_{t}^2}{2}\nonumber                                                                          \\
                                & \leq \srb{\frac{M g(t) \eta_t}{2\epsilon}}^{\frac{1}{1-r}}+\frac{M \eta_{t}^2}{2}   \nonumber                              \\
                                & \leq \srb{\srb{\frac{M g(t)}{2\epsilon}}^{\frac{1}{1-r}} +\frac{M}{2}} \eta_{t}^{\min\{\frac{1}{1-r},2\}}                  \\
                                & \leq   \srb{\srb{\frac{M g(t)}{2\epsilon}}^{\frac{1}{1-r}} +\frac{M}{2}} \eta_{t}^{k}.\nonumber
              \end{align}
        \item
              \label{case:bad_progress.weak} Suppose that~\eqref{eq.smaller.weak} does not hold. Let $\Rfwt\in \{\fwt, \fwt+1 \ldots, t\}$ be the smallest index such that
              \begin{align*}
                  \subopt_{i} \geq \srb{\frac{M g(i)\eta_{i}}{2\epsilon}}^{\frac{1}{1-r}}
              \end{align*}
              for all $i\in\{\Rfwt, \Rfwt + 1, \ldots, t\}$.
              Then, Equality~\eqref{eq.gap}, the weak $(M, r)$-growth property, and $\gap_i \geq \subopt_{i}$ for all $i\in\N$ imply that
              \begin{align}\label{eq.rec.weak}
                  \subopt_{i+1} \leq \subopt_{i}+\eta_{i}\gap_{i}\srb{\frac{M\eta_{i}}{2\subopt_{i}^{1-r}}-1}
                  \leq  \subopt_{i} \left(1 - \srb{1-\frac{\epsilon}{g(i)}}\eta_{i}\right)
              \end{align}
              for all $i\in \{\Rfwt, \Rfwt + 1, \ldots, t\}$.
              By repeatedly applying~\eqref{eq.rec.weak}, by~\eqref{eq.main_technical.prod_sum_advanced}, and by A2,
              \begin{align}\label{eq.calc_big.weak}
                  \subopt_{t+1} \leq \subopt_{\Rfwt}\prod_{i=\Rfwt}^{t} \left(1 - \srb{1-\frac{\epsilon}{g(i)}}\eta_i\right) & \leq \subopt_\Rfwt \left(\frac{\eta_{t}}{\eta_{\Rfwt-1}}\right)^{ g(\Rfwt) - \epsilon}.
              \end{align}
              We distinguish between two subcases:
              \begin{enumerate}
                  \item\label{case:weak_r_a} $\Rfwt = \fwt$. In this case,~\eqref{eq.calc_big.weak} implies that
                        \begin{align}\label{eq.fwt_does_not_exist.weak}
                            \subopt_{t+1} & \leq \subopt_{\fwt} \left(\frac{\eta_{t}}{\eta_{\fwt-1}}\right)^{ g(\fwt) - \epsilon}.
                        \end{align}
                  \item \label{case:weak_r_b} $\Rfwt > \fwt$. In this case, $\subopt_{\Rfwt-1}\leq \left(\frac{Mg(\Rfwt-1)\eta_{\Rfwt-1}}{2\epsilon}\right)^{\frac{1}{1-r}}$. We apply~\eqref{eq.good_growth.weak} in Case~\ref{case:good_progress.weak} to get $\subopt_{\Rfwt}\leq \srb{\srb{\frac{M g(\Rfwt-1)}{2\epsilon}}^{\frac{1}{1-r}} +\frac{M}{2}}\eta_{\Rfwt-1}^{\min\{\frac{1}{1-r},2\}}$.
                        Thus,~\eqref{eq.calc_big.weak} gives
                        \begin{align*}
                            \subopt_{t+1} & \leq \srb{\srb{\frac{M g(\Rfwt-1)}{2\epsilon}}^{\frac{1}{1-r}} +\frac{M}{2}} \left(\frac{\eta_{t}}{\eta_{\Rfwt-1}}\right)^{ g(\Rfwt) - \epsilon} \eta_{\Rfwt-1}^{\min\{\frac{1}{1-r},2\}}.
                        \end{align*}
                        Taking the worst possible $\Rfwt > \fwt$ gives
                        \begin{align*}
                            &\subopt_{t+1} \\
                            & \leq \max_{\Rfwt\in\{\fwt, \fwt+1, \ldots, t\}} \srb{\srb{\frac{M g(\Rfwt-1)}{2\epsilon}}^{\frac{1}{1-r}} +\frac{M}{2}} \left(\frac{\eta_{t}}{\eta_{\Rfwt-1}}\right)^{ g(\Rfwt) - \epsilon} \eta_{\Rfwt-1}^{\min\{\frac{1}{1-r},2\}}                                                                                \\
                                          & \leq  \max_{\Rfwt\in\{\fwt, \fwt+1, \ldots, t\}} \srb{\srb{\frac{M g(t)}{2\epsilon}}^{\frac{1}{1-r}} +\frac{M}{2}} \left(\frac{\eta_{t}}{\eta_{\Rfwt-1}}\right)^{ g(\Rfwt) - \epsilon} \eta_{\Rfwt-1}^{\min\{\frac{1}{1-r},2\}}       & \text{$\triangleright$ by~A1}                                \\
                                          & \leq  \srb{\srb{\frac{M g(t)}{2\epsilon}}^{\frac{1}{1-r}} +\frac{M}{2}} \eta_{t}^{k}.                                                                                                                                                               
                        \end{align*}
              \end{enumerate}
    \end{enumerate}
    Cases~\ref{case:good_progress.weak},~\ref{case:weak_r_a}, and~\ref{case:weak_r_b} imply~\eqref{eq.rate_weak}.
\end{proof}
For the weak $(M, r)$-growth setting, Theorem~\ref{thm:rate-weak} shows that the convergence rates are capped by $\tilde\cO(t^{-2})$, irrespective of the choice of $g$. Thus, the log-adaptive step-size converges at the same rate as fixed-$\ell$ open-loop step-sizes, up to polylogarithmic factors.

\section{Numerical experiments}\label{sec.experiments}

In this section, we present some numerical experiments. The experiments are implemented in \textsc{Python} and performed on an AMD EPYC 7502P 32-Core Processor CPU with 64GB of RAM. Our code is publicly available on
\href{https://github.com/ZIB-IOL/open_loop_fast}{GitHub}. We adopt a similar experimental setup to \cite{wirth2023accelerated} for our study.  
We compare the optimality measures $\gap_t$, $\primaldual_t$, and $\subopt_t$ of \fw{} with $\eta_t = \frac{g(t)}{t+g(t)}$ for $g(t) = 2, g(t)=4,$ and $g(t)=2+\log(t+1)$ for several instances of two classes of problems: constrained regression in Section~\ref{sec:constrained regression} and collaborative filtering in Section~\ref{sec.collaborative_filtering}. For additional visualization, we also plot $\cO(t^{-2})$.

\subsection{Regression}\label{sec:constrained regression}

We compare the different step-sizes for the constrained regression problem
\begin{align}\label{eq.regression}
    \min_{\xx\in\R^n, \|\xx\|_p \leq \beta} & \; \frac{1}{2}\|A \xx - \yy\|_2^2,
\end{align}
where $A\in\R^{m \times n}$, $\yy\in\R^m$, $p>1$, and $\beta > 0$. We denote the unconstrained optimizer by $\xx_{\text{unc}}:= \argmin_{\xx\in\R^n} \frac{1}{2}\|A \xx - \yy\|_2^2$. For different $\ell_p$-balls, $p\in\{2, 5\}$, and for $\beta \in\|\xx_{\text{unc}}\|_p\cdot \{1/2,3/2\}$, we compare the convergence rates. These values of $\beta$ correspond to the location of the unconstrained optimizer lying in the relative exterior, and in the relative interior of the feasible region, respectively. As it is discussed in~\cite{wirth2023accelerated}, when    $\beta = 1/2$, problem~\eqref{eq.regression} satisfies the $(M,r)$ strong growth property for $r = 2/\max(2,p)$ some  constant $M>0$ that depends on $A, \xx_{\text{unc}},$ and $p$.
On the other hand, again as discussed in~\cite{wirth2023accelerated}, when
$\beta =3/2$, problem~\eqref{eq.regression} satisfies the $(M,0)$ strong growth property and the $(M,r)$ weak growth property for $r = 2/\max(2,p)$ and some  constant $M>0$ that depends on $A, \xx_{\text{unc}},$ and $p$.

\begin{figure}[!th]
    \captionsetup[subfigure]{justification=centering}
    \begin{tabular}{c c c}    
        \begin{subfigure}{.31\textwidth}
            \centering
            \includegraphics[width=1\textwidth]{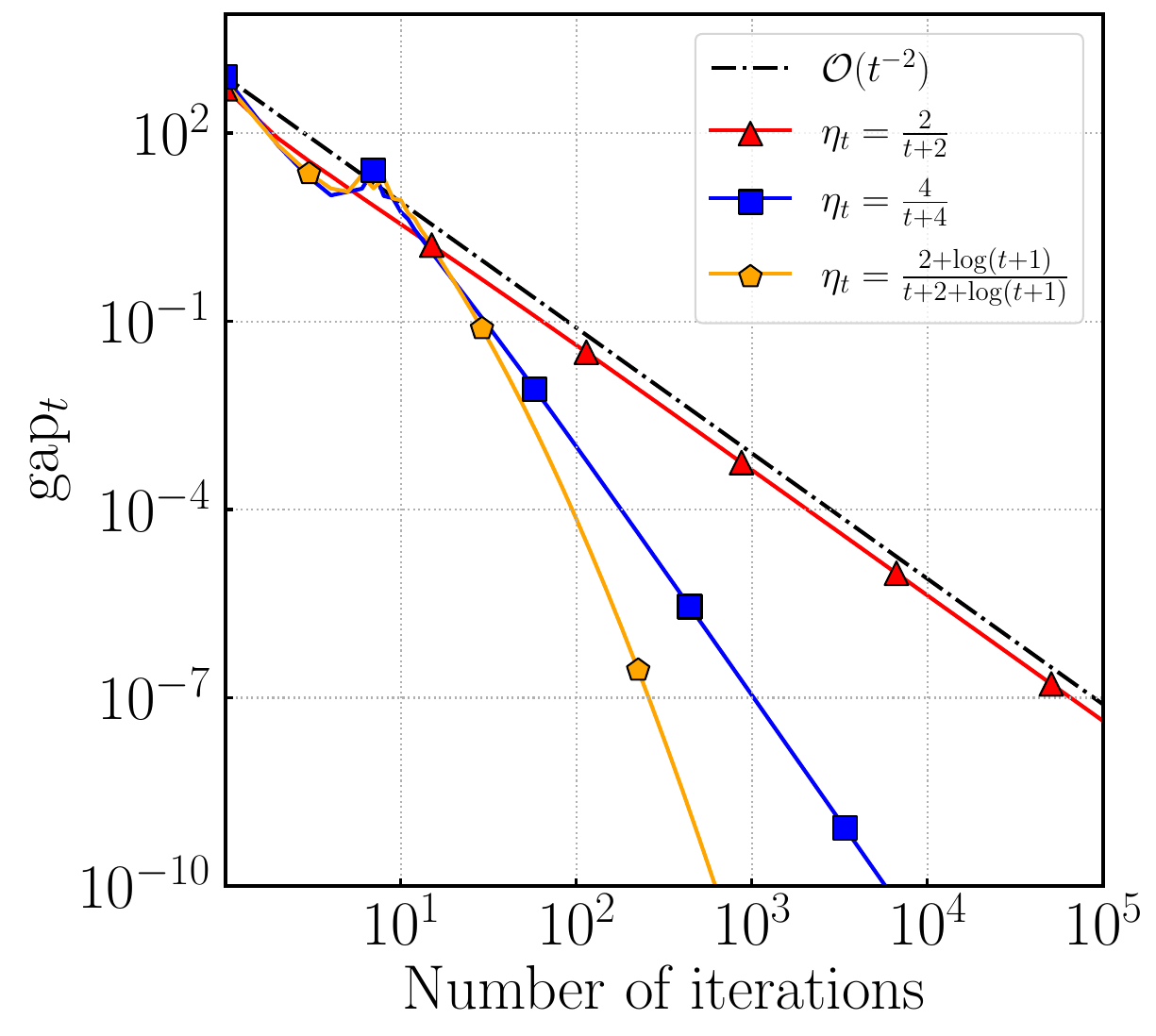}
            \caption{relative exterior, $\gap_t$.}\label{fig:reg_boston_gap_2_exterior}
        \end{subfigure} &
        \begin{subfigure}{.31\textwidth}
            \centering
            \includegraphics[width=1\textwidth]{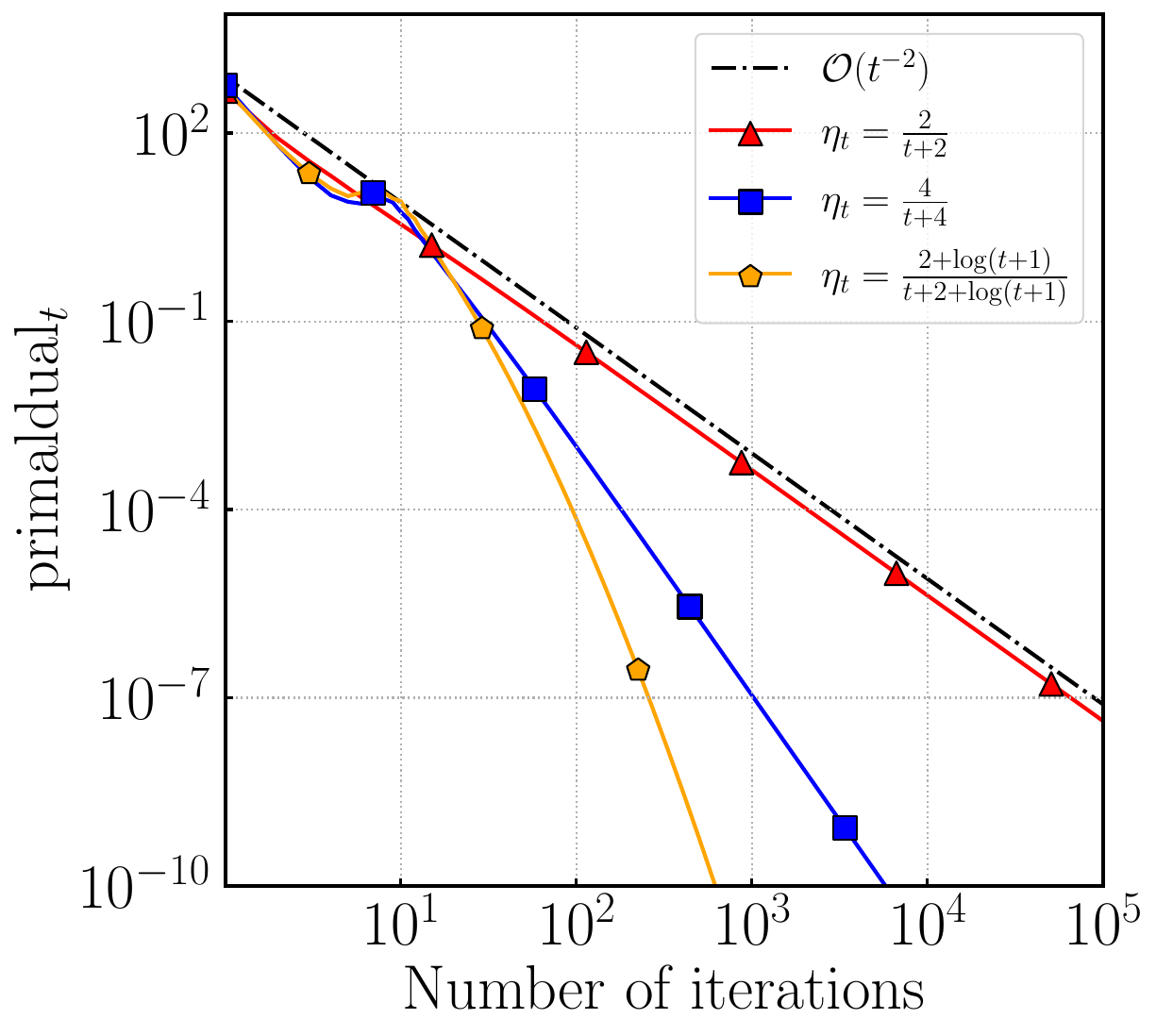}
            \caption{relative exterior, $\primaldual_t$.}\label{fig:reg_boston_primaldual_2_exterior}
        \end{subfigure} &
        \begin{subfigure}{.31\textwidth}
            \centering
            \includegraphics[width=1\textwidth]{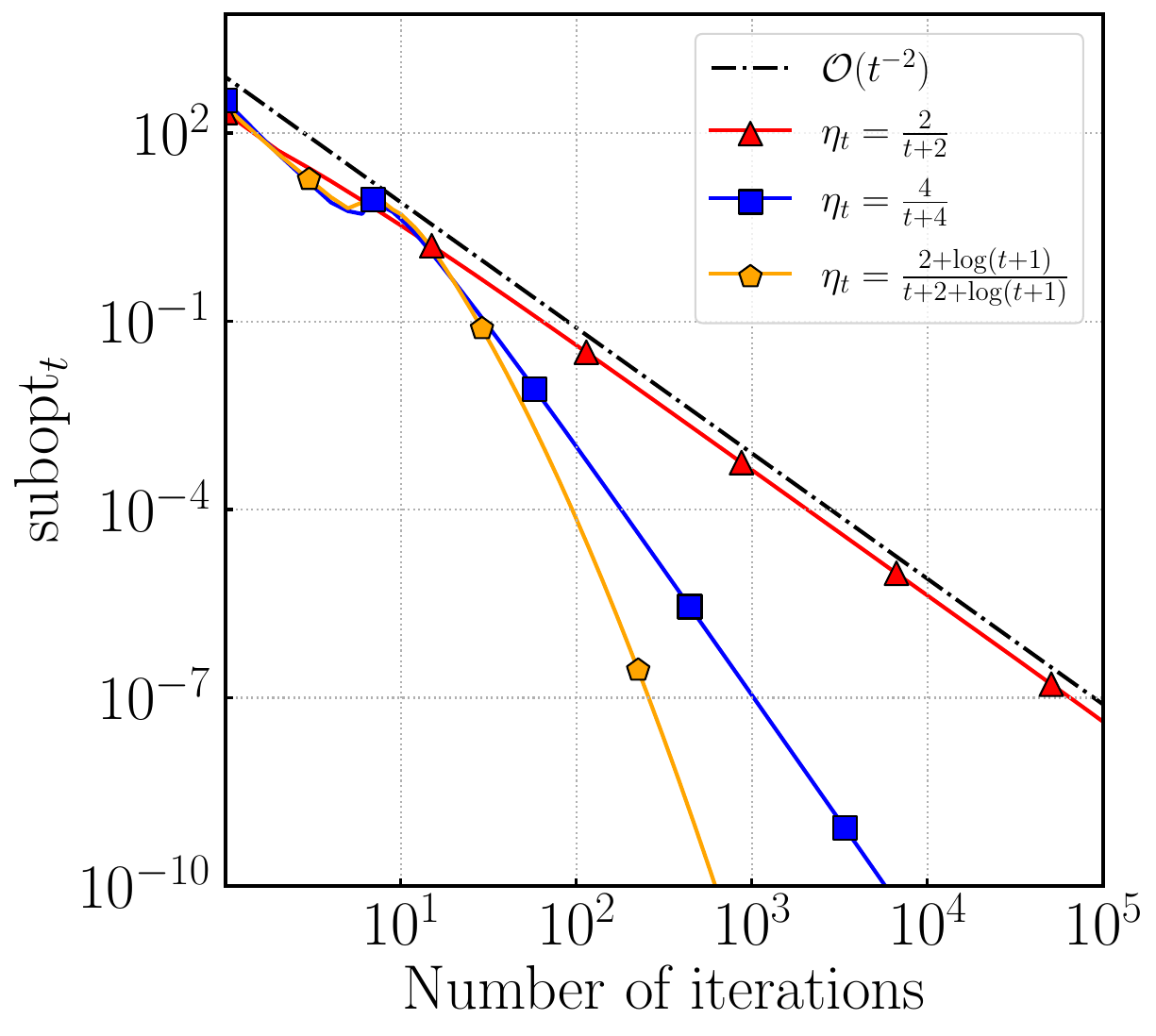}
            \caption{relative exterior, $\subopt_t$.}\label{fig:reg_boston_subopt_2_exterior}
        \end{subfigure}\\    
        \begin{subfigure}{.31\textwidth}
            \centering
            \includegraphics[width=1\textwidth]{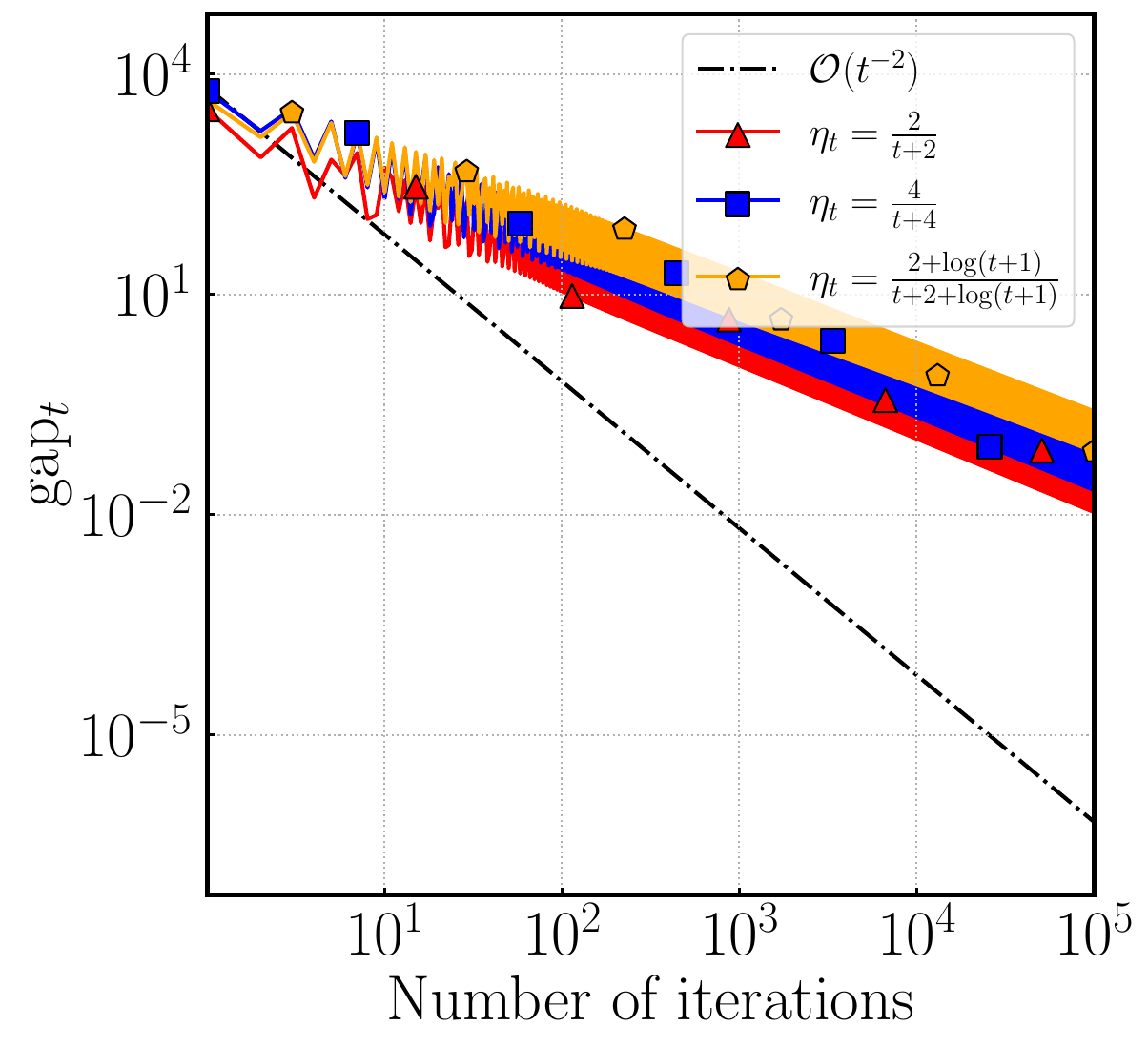}
            \caption{relative interior, $\gap_t$.}\label{fig:reg_boston_gap_2_interior}
        \end{subfigure} &
        \begin{subfigure}{.31\textwidth}
            \centering
            \includegraphics[width=1\textwidth]{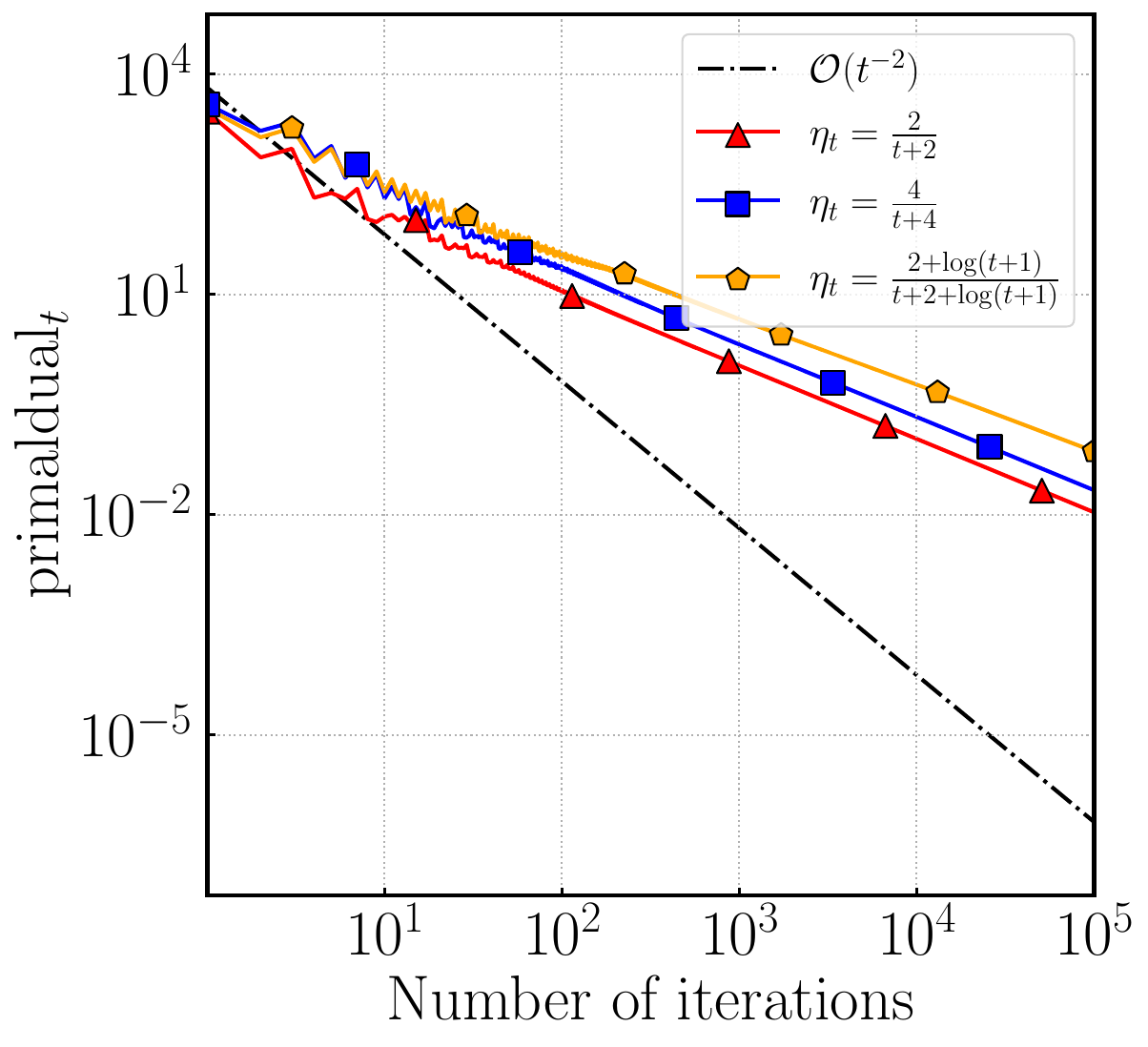}
            \caption{relative interior, $\primaldual_t$.}\label{fig:reg_boston_primaldual_2_interior}
        \end{subfigure} &
        \begin{subfigure}{.31\textwidth}
            \centering
            \includegraphics[width=1\textwidth]{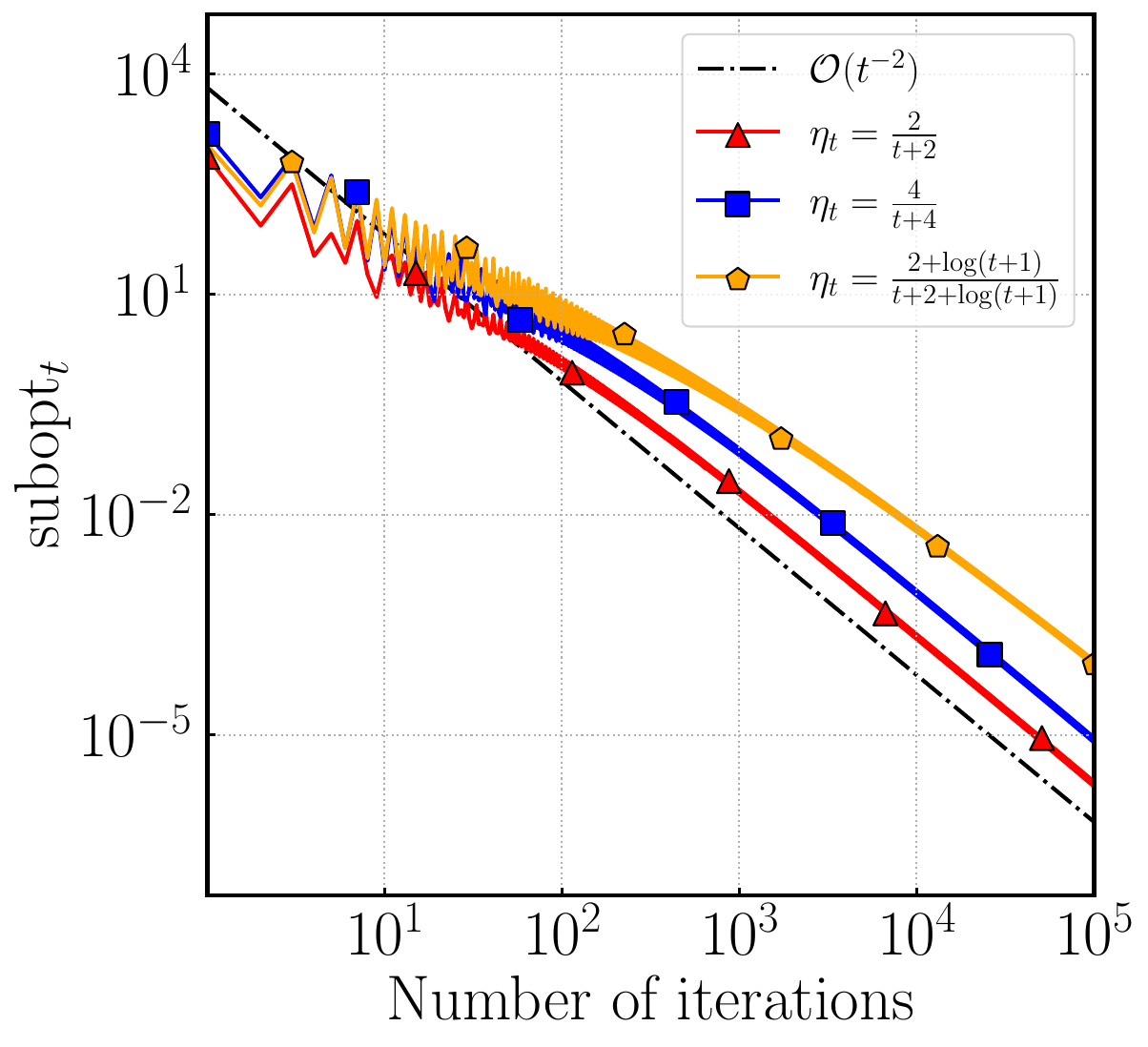}
            \caption{relative interior, $\subopt_t$.}\label{fig:reg_boston_subopt_2_interior}
        \end{subfigure}   
    \end{tabular}
    \caption{\textbf{Constrained regression over the $\ell_2$-ball.}
        Convergence rate comparison of \fw{} with different step-sizes for \eqref{eq.regression} for different locations of the unconstrained optimizer in the relative interior, on the relative boundary, and in the relative exterior of the feasible region and for the three different optimality measures $\gap_t$, $\primaldual_t$, and $\subopt_t$  on the Boston-housing dataset. Axes are in log scale.
    }\label{fig:reg_boston_2}
\end{figure}

\begin{figure}[!t]
    \captionsetup[subfigure]{justification=centering}
    \begin{tabular}{c c c}    
        \begin{subfigure}{.31\textwidth}
            \centering
            \includegraphics[width=1\textwidth]{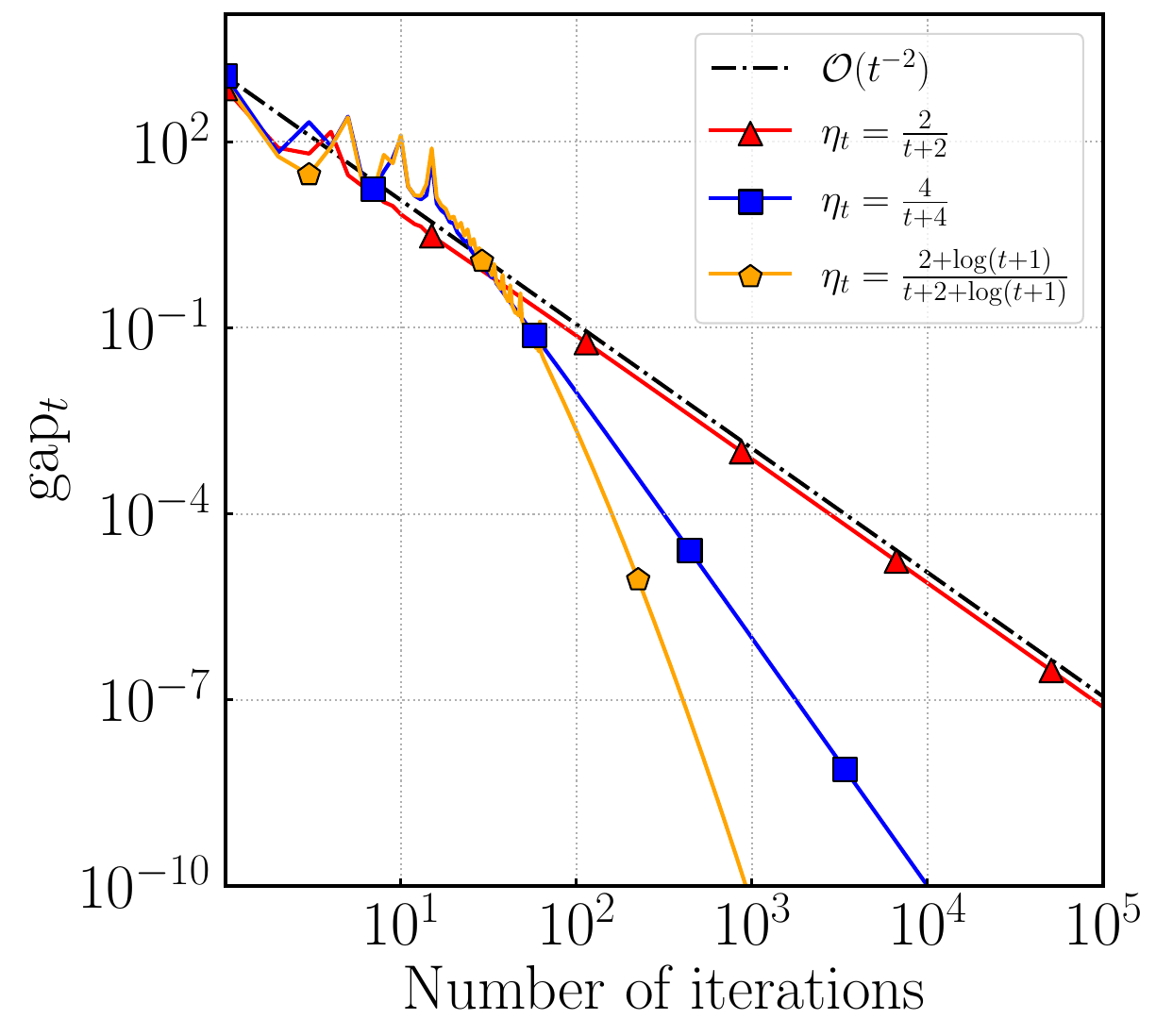}
            \caption{relative exterior, $\gap_t$.}\label{fig:reg_boston_gap_5_exterior}
        \end{subfigure} &
        \begin{subfigure}{.31\textwidth}
            \centering
            \includegraphics[width=1\textwidth]{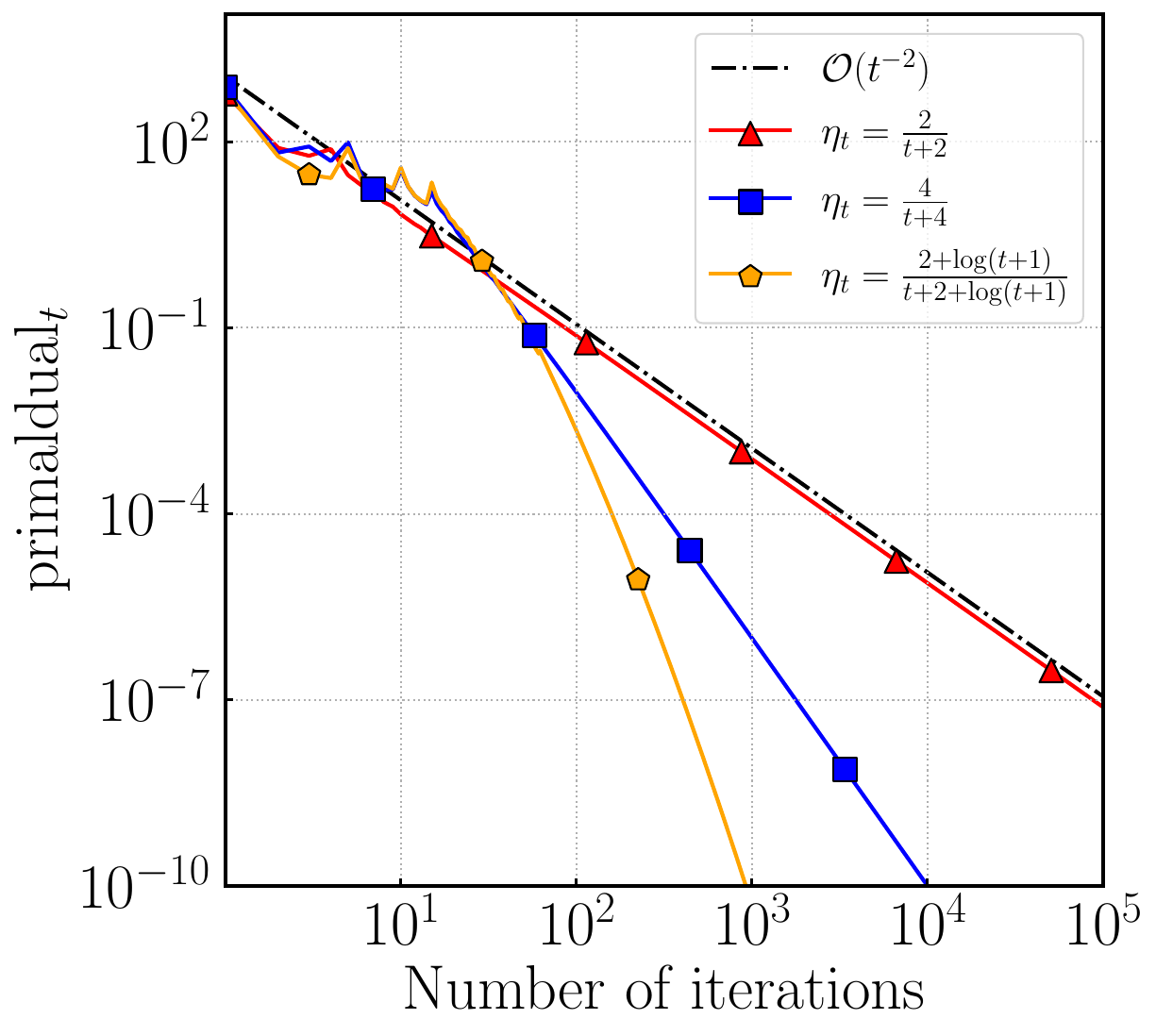}
            \caption{relative exterior, $\primaldual_t$.}\label{fig:reg_boston_primaldual_5_exterior}
        \end{subfigure} &
        \begin{subfigure}{.31\textwidth}
            \centering
            \includegraphics[width=1\textwidth]{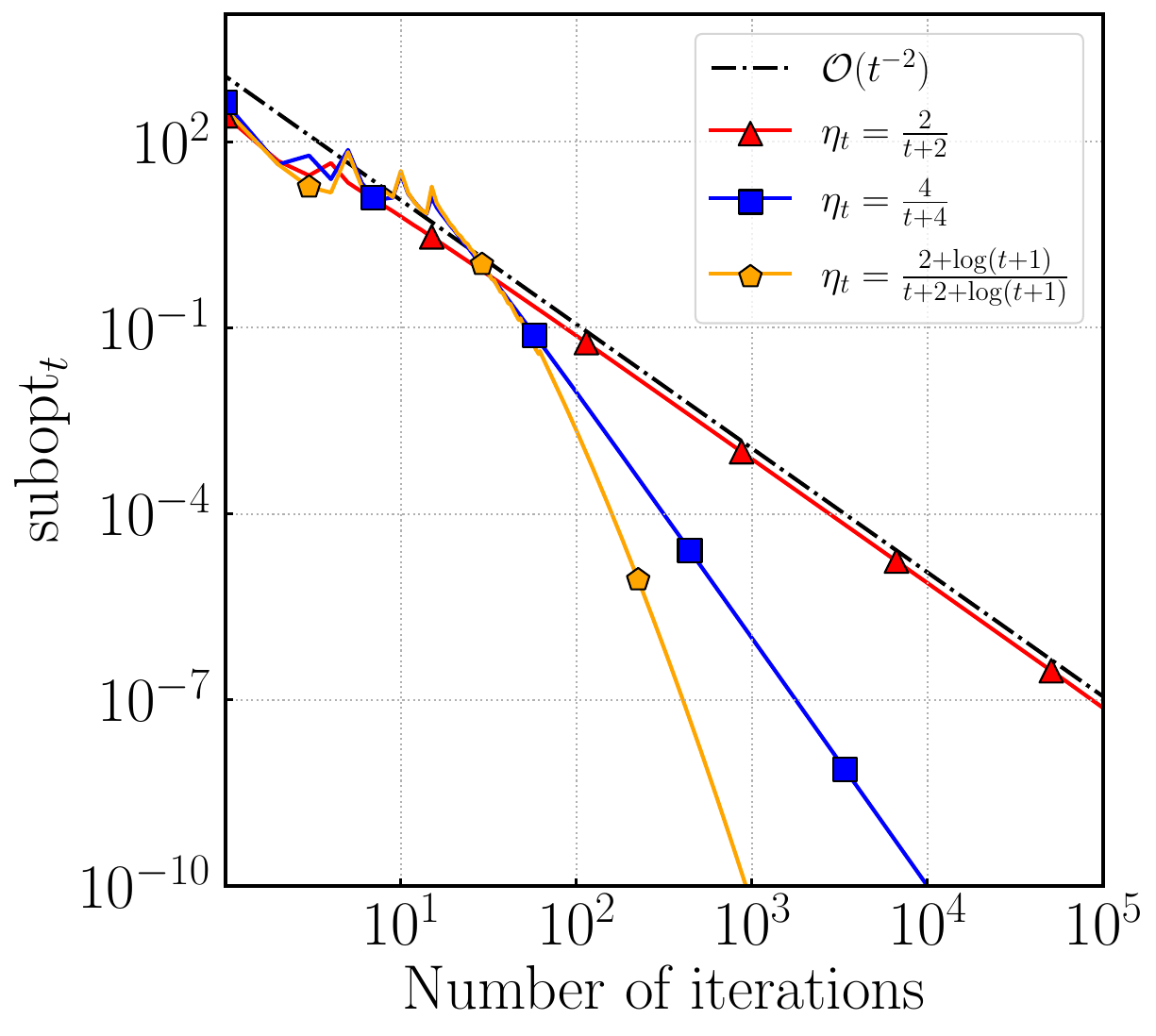}
            \caption{relative exterior, $\subopt_t$.}\label{fig:reg_boston_subopt_5_exterior}
        \end{subfigure}\\    
        \begin{subfigure}{.31\textwidth}
            \centering
            \includegraphics[width=1\textwidth]{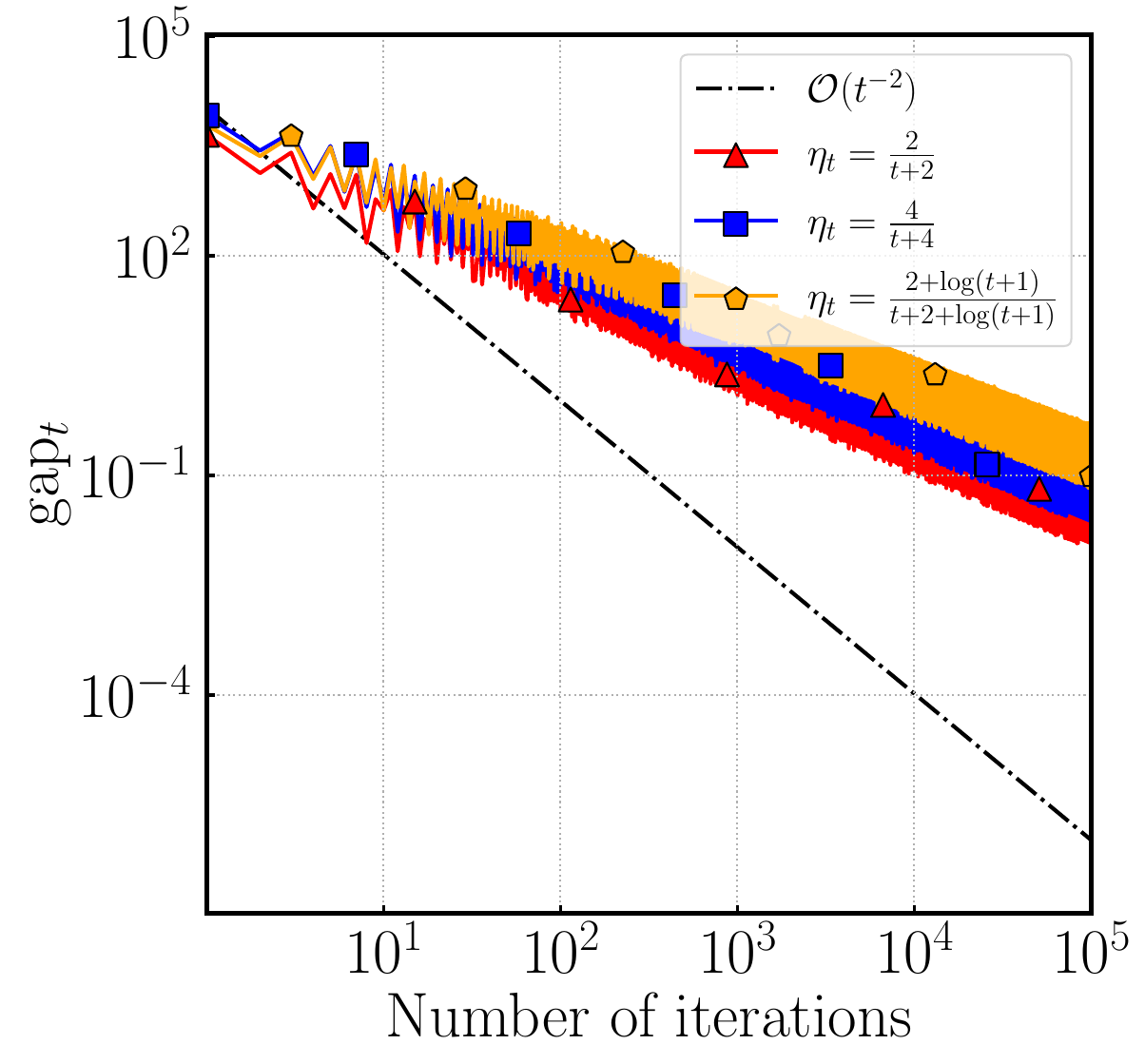}
            \caption{relative interior, $\gap_t$.}\label{fig:reg_boston_gap_5_interior}
        \end{subfigure} &
        \begin{subfigure}{.31\textwidth}
            \centering
            \includegraphics[width=1\textwidth]{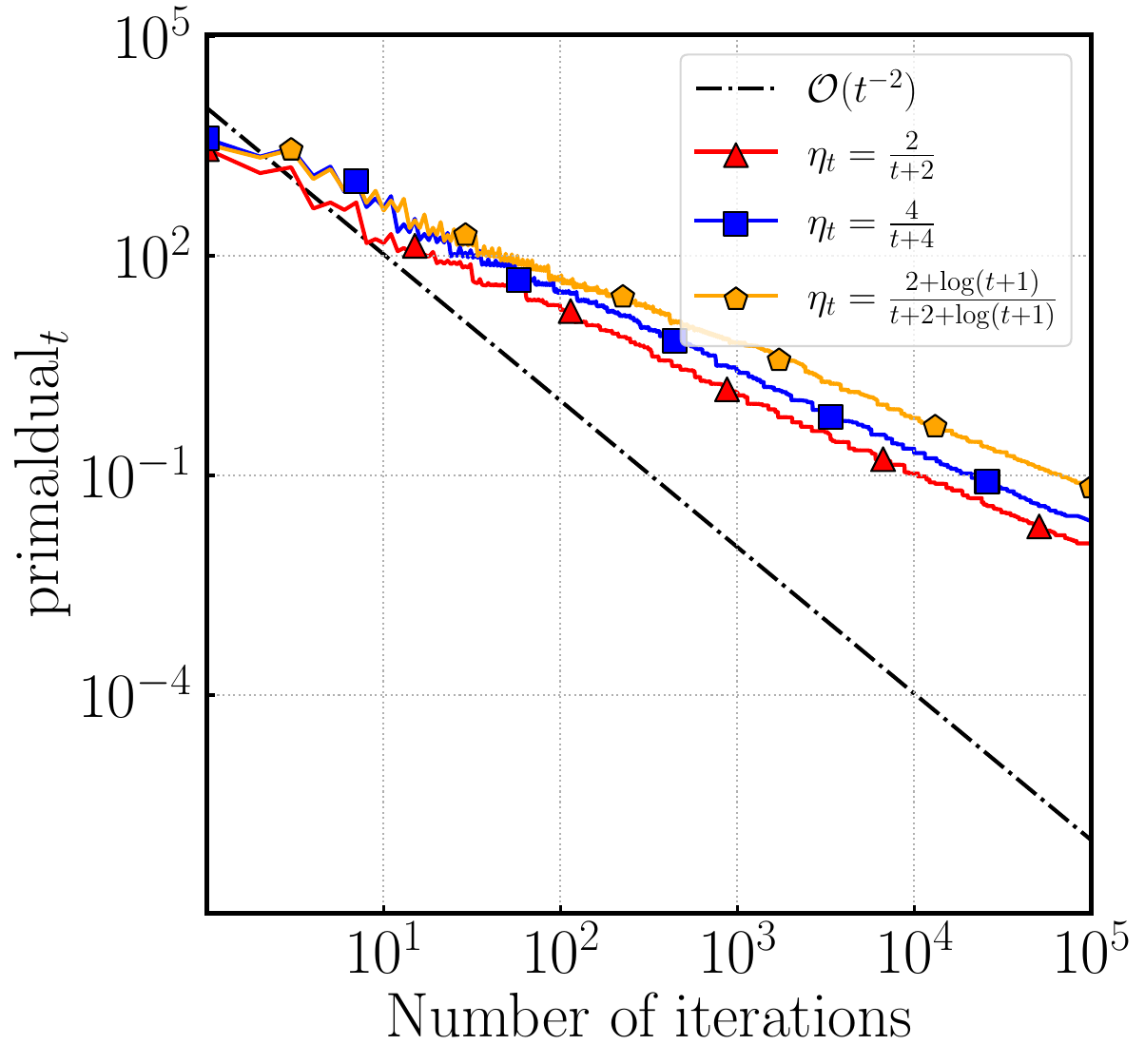}
            \caption{relative interior, $\primaldual_t$.}\label{fig:reg_boston_primaldual_5_interior}
        \end{subfigure} &
        \begin{subfigure}{.31\textwidth}
            \centering
            \includegraphics[width=1\textwidth]{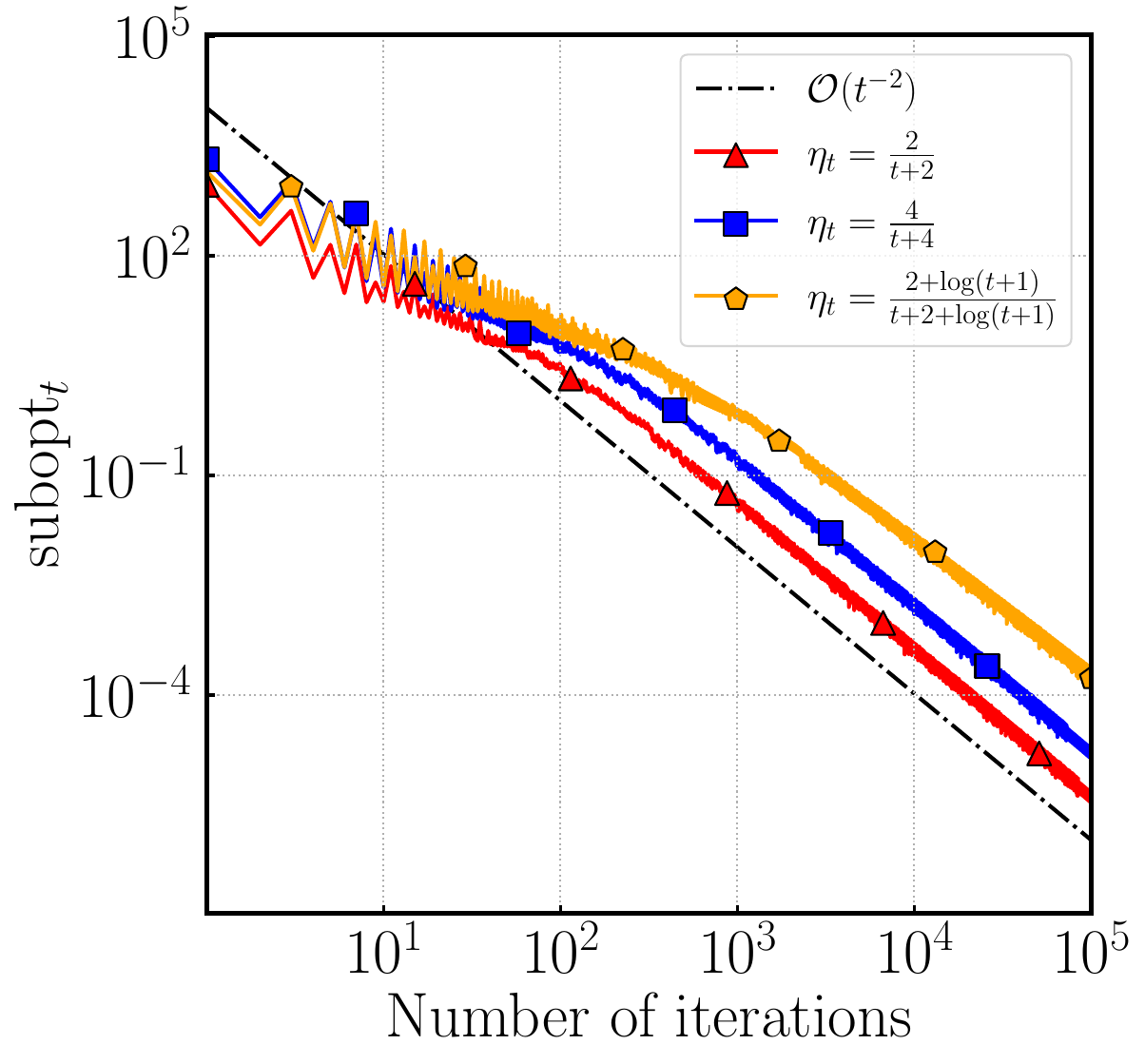}
            \caption{relative interior, $\subopt_t$.}\label{fig:reg_boston_subopt_5_interior}
        \end{subfigure}   
    \end{tabular}
    \caption{\textbf{Constrained regression over the $\ell_5$-ball.}
        Convergence rate comparison of \fw{} with different step-sizes for \eqref{eq.regression} for different locations of the unconstrained optimizer in the relative interior, on the relative boundary, and in the relative exterior of the feasible region and for the three different optimality measures $\gap_t$, $\primaldual_t$, and $\subopt_t$ on the Boston-housing dataset. Axes are in log scale.
    }\label{fig:reg_boston_5}
\end{figure}

To construct instances of~\eqref{eq.regression}, the matrix $A$ and vector $\yy$ are obtained by Z-score normalizing the data of the Boston-housing dataset ($m = 506, n=13$), obtained via the \textsc{sklearn} package, see, e.g., \href{https://scikit-learn.org/0.15/index.html}{https://scikit-learn.org/0.15/index.html}.  Numerical results are presented in Figures~\ref{fig:reg_boston_2} and~\ref{fig:reg_boston_5}.

In all settings and for all optimality measures, the proposed step-size $\eta_t=\frac{2+\log(t+1)}{t+2+\log(t+1)}$ indeed converges at least as fast as fixed-$\ell$ step-sizes of the form $\eta_t=\frac{\ell}{t+\ell}$ for any $\ell\in\N$, up to polylogarithmic factors.
Furthermore, in the strong $(M,r)$-growth settings of Figures~\ref{fig:reg_boston_gap_2_exterior}--\ref{fig:reg_boston_subopt_2_exterior} and Figures~\ref{fig:reg_boston_gap_5_exterior}--\ref{fig:reg_boston_subopt_5_exterior} and for all optimality measures, the log-adaptive step-size admits arbitrarily fast sublinear convergence rates and strictly outperforms the two fixed-$\ell$ step-sizes as predicted by Theorem~\ref{thm:strong_M_r_growth}.

\subsection{Collaborative filtering}\label{sec.collaborative_filtering}

We compare the different step-sizes for the collaborative filtering problem as in \cite{mehta2007robust}
\begin{align}\label{eq.collaborative_filterting}
    \min_{X\in\R^{m\times n}, \|X\|_{\nuc}\leq \beta} & \; \frac{1}{|\cI|} \sum_{(i,j)\in\cI} H(A_{i,j} - X_{i,j}),
\end{align}
where $H$ is a Huber loss \cite{huber1992robust}, that is,
\begin{align*}
    H\colon x\in \R \mapsto  \begin{cases}
        \frac{x^2}{2},           & \text{if} \ |x| \leq 1 \\
        \rho(|x| - \frac{1}{2}), & \text{if} \ |x| > 1,
    \end{cases}
\end{align*}
$A\in\R^{m\times n}$ is a matrix with missing entries, that is, only the entries $A_{i,j}$ with $(i,j)\in\cI$ are known, where $\cI \subseteq \{1,\ldots,m\}\times\{1,\ldots,n\}$ is the set of observed indices,
$\|\cdot\|_{\nuc}$ is the nuclear norm, and $\beta>0$.

For $\beta \in\{1000, 3000\}$, we compare the convergence rates on the movielens dataset ($m=943, n=1682$). The dataset is obtained via \href{https://grouplens.org/datasets/movielens/100k/}{https://grouplens.org/datasets/movielens/100k/}

\begin{figure}[!t]
    \captionsetup[subfigure]{justification=centering}
    \begin{tabular}{c c c}
        \begin{subfigure}{.31\textwidth}
            \centering
            \includegraphics[width=1\textwidth]{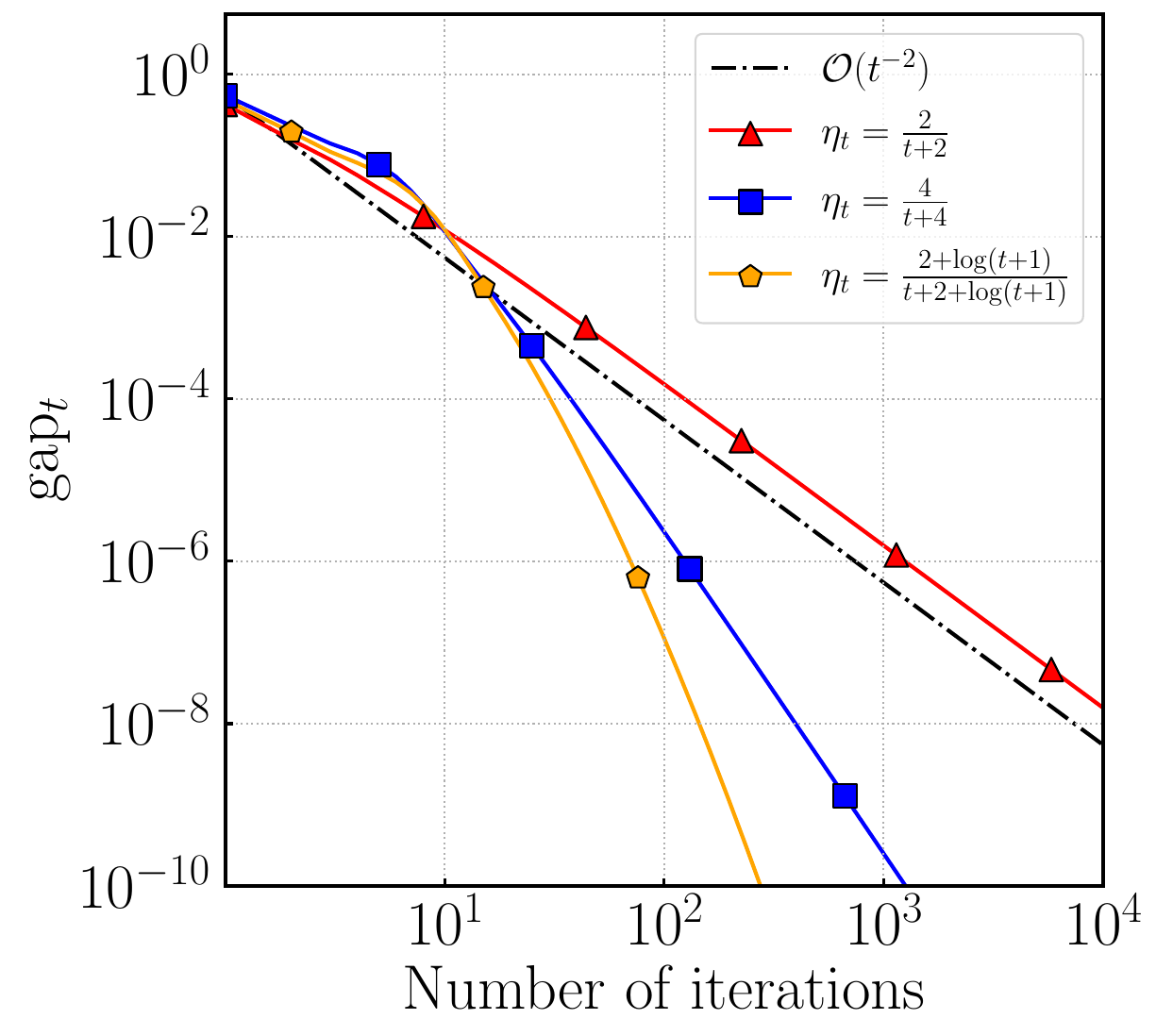}
            \caption{radius $1000$, $\gap_t$.}\label{fig:col_fil_1000_gap}
        \end{subfigure} &
        \begin{subfigure}{.31\textwidth}
            \centering
            \includegraphics[width=1\textwidth]{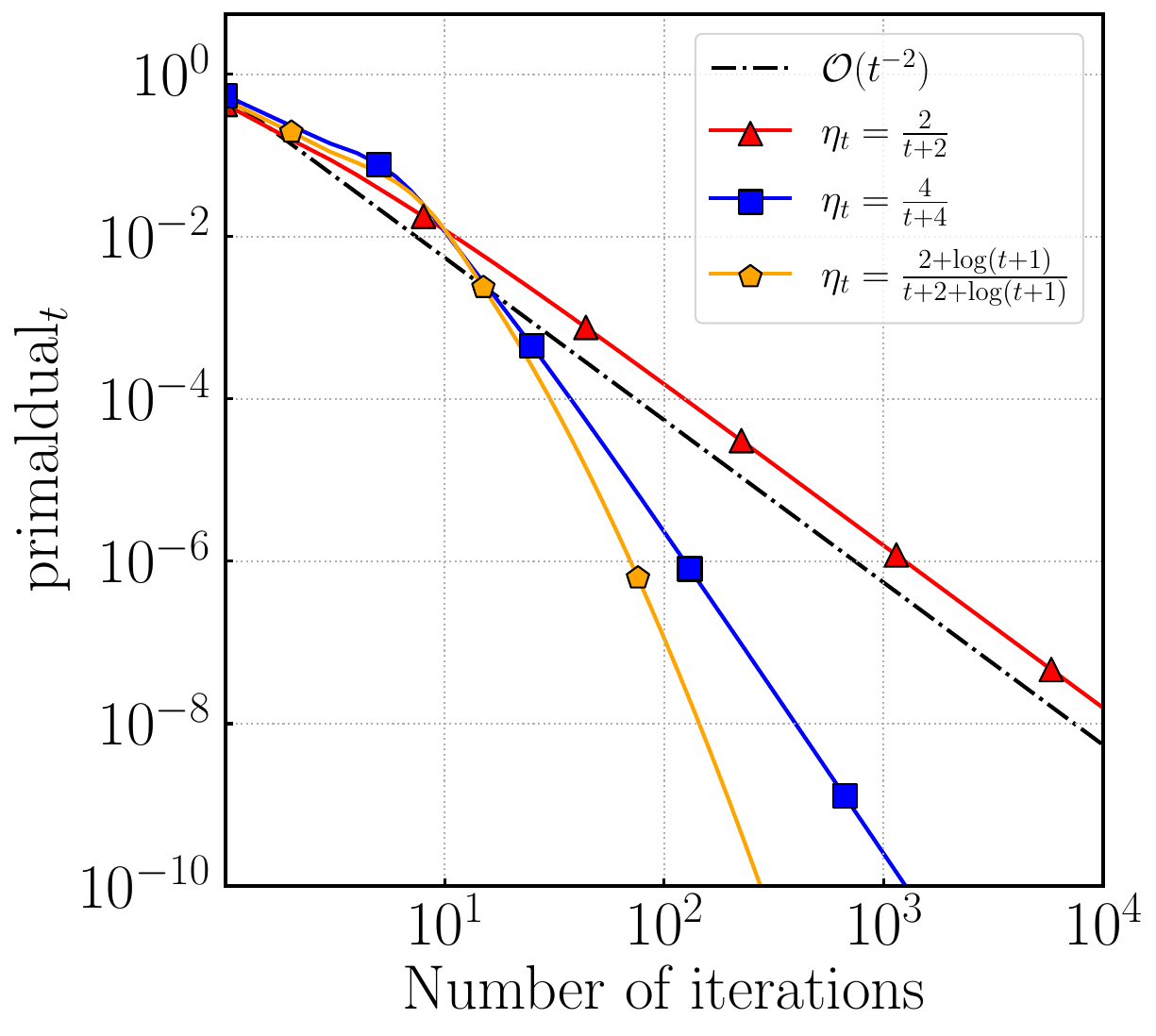}
            \caption{radius $1000$, $\primaldual_t$.}\label{fig:col_fil_1000_primaldual}
        \end{subfigure} &
        \begin{subfigure}{.31\textwidth}
            \centering
            \includegraphics[width=1\textwidth]{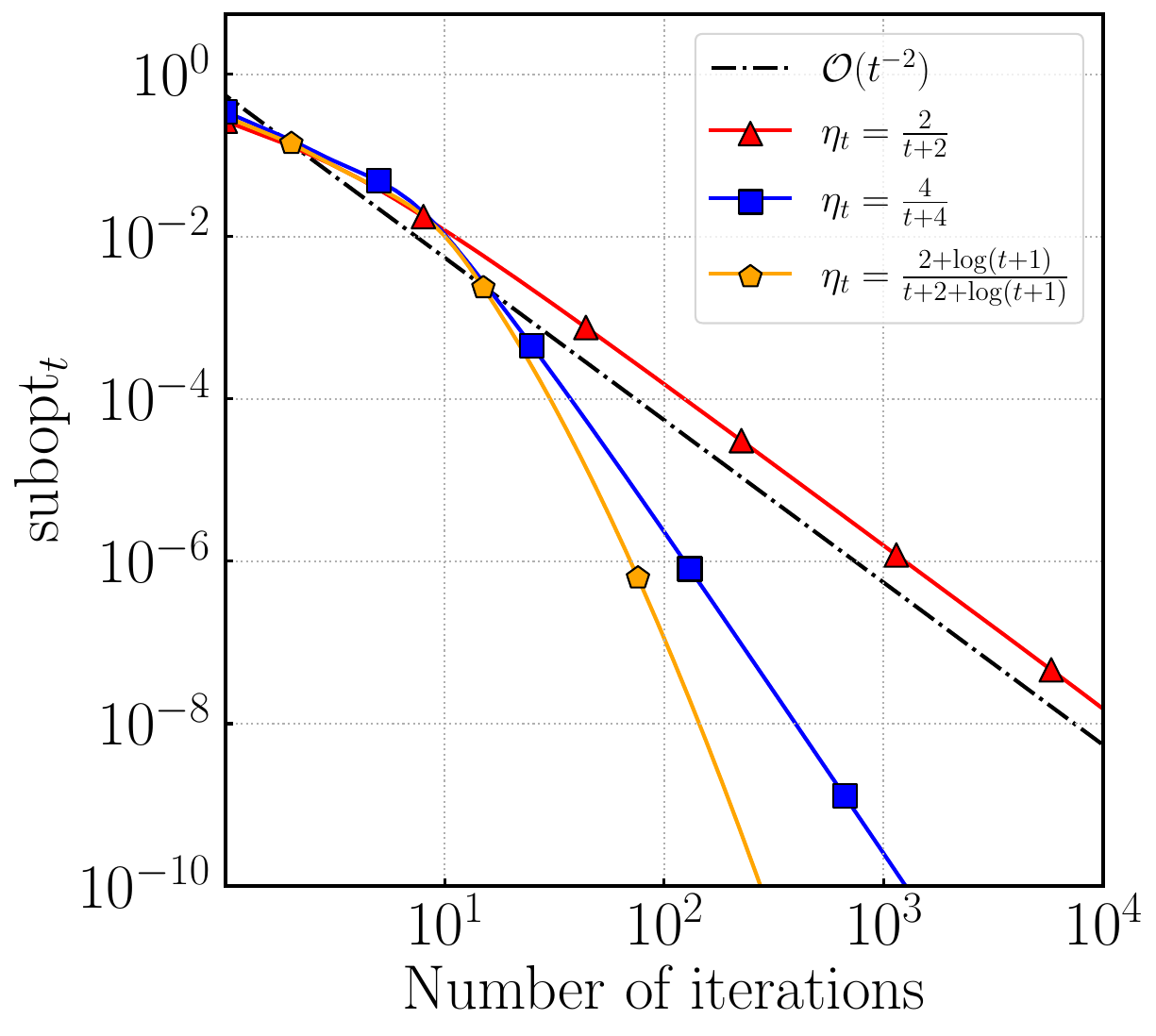}
            \caption{radius $1000$, $\subopt_t$.}\label{fig:col_fil_1000_subopt}
        \end{subfigure}   \\
        \begin{subfigure}{.31\textwidth}
            \centering
            \includegraphics[width=1\textwidth]{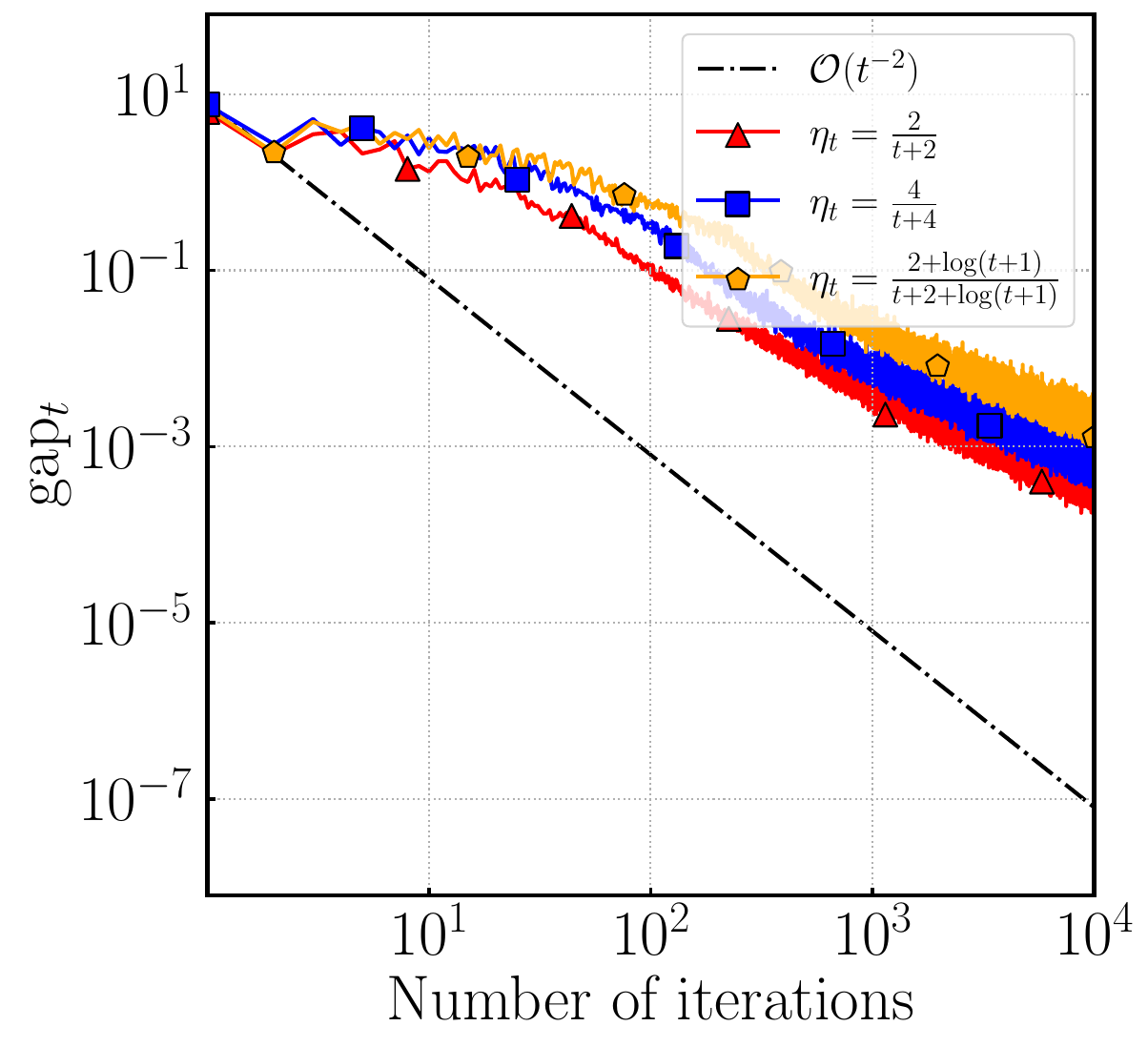}
            \caption{radius $3000$, $\gap_t$.}\label{fig:col_fil_3000_gap}
        \end{subfigure} &
        \begin{subfigure}{.31\textwidth}
            \centering
            \includegraphics[width=1\textwidth]{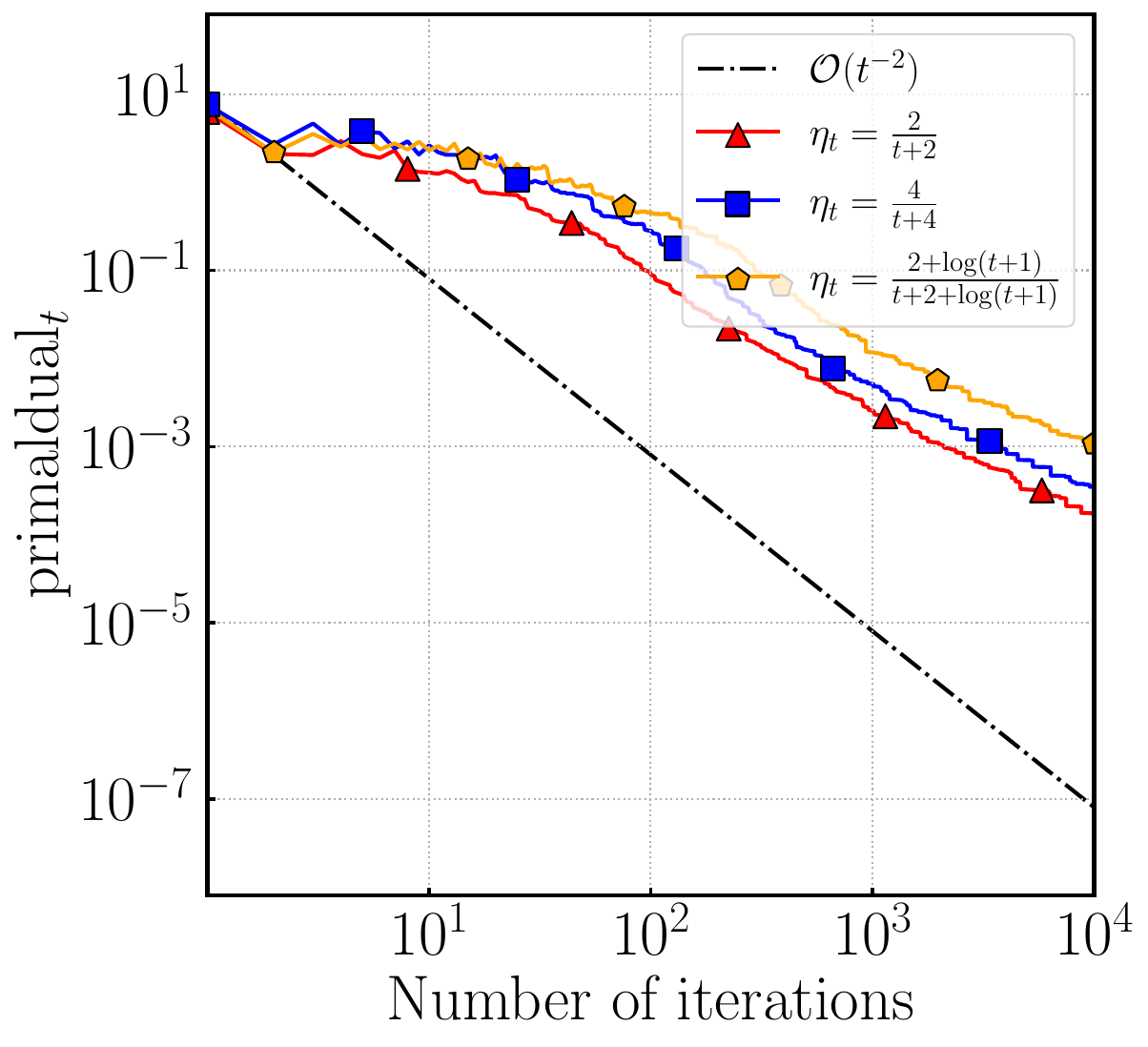}
            \caption{radius $3000$, $\primaldual_t$.}\label{fig:col_fil_3000_primaldual}
        \end{subfigure} &
        \begin{subfigure}{.31\textwidth}
            \centering
            \includegraphics[width=1\textwidth]{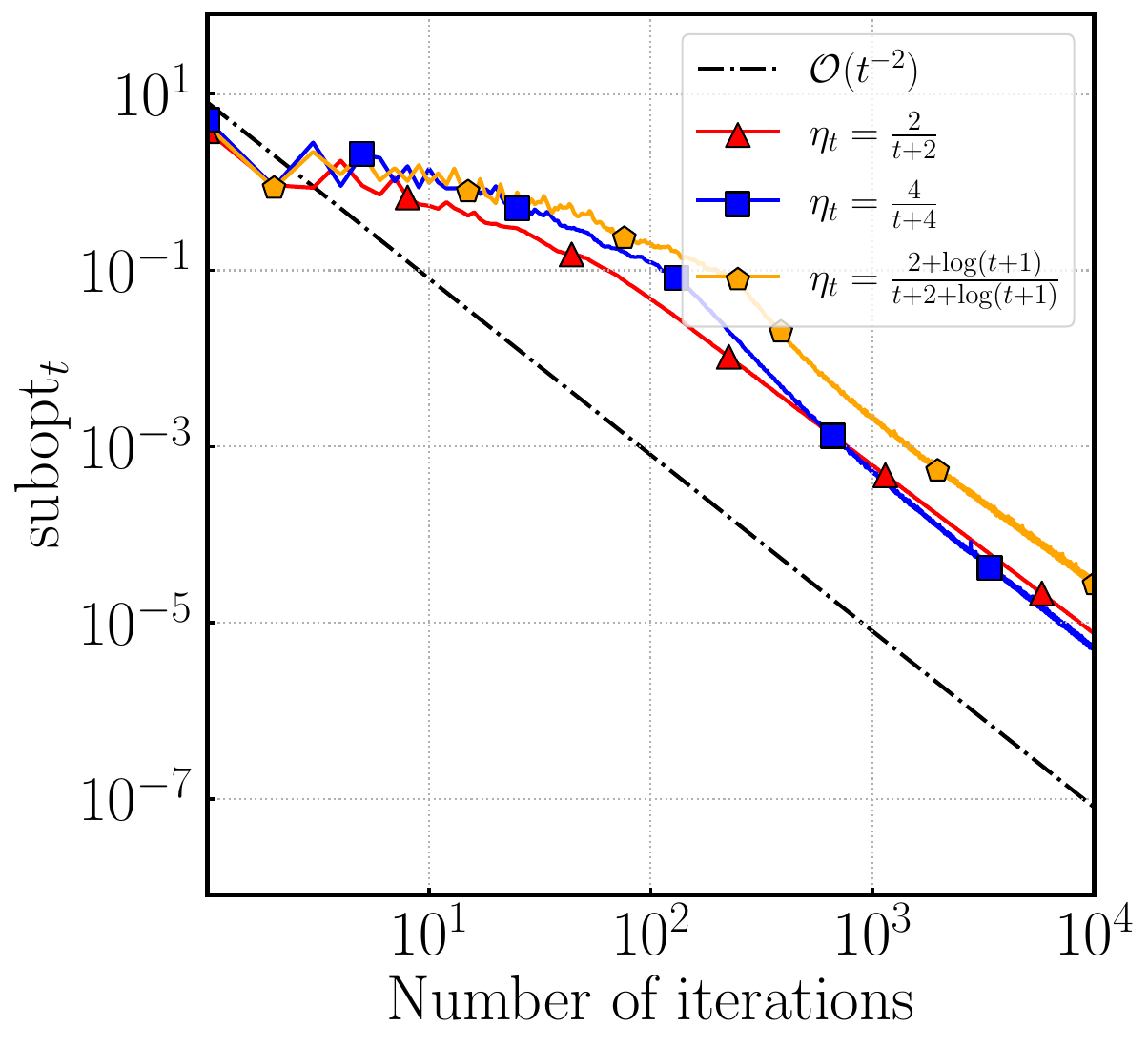}
            \caption{radius $3000$, $\subopt_t$.}\label{fig:col_fil_3000_subopt}
        \end{subfigure}
    \end{tabular}
    \caption{\textbf{Collaborative filtering over nuclear norm balls of radii $1000$ and $3000$.}
        Convergence rate comparison of \fw{} with different step-sizes for \eqref{eq.collaborative_filterting} over nuclear norm balls of radii $1000$ and $3000$ for the three different optimality measures $\gap_t$, $\primaldual_t$, and $\subopt_t$ on the movielens dataset. Axes are in log scale.
    }\label{fig:col_fil}
\end{figure}

See \cite{harper2015movielens}. The results are presented in Figure~\ref{fig:col_fil}.
For both radii and all optimality measures, the proposed step-size $\eta_t=\frac{2+\log(t+1)}{t+2+\log(t+1)}$ indeed converges at least as fast as fixed-$\ell$ step-sizes of the form $\eta_t=\frac{\ell}{t+\ell}$ for any $\ell\in\N$, up to polylogarithmic factors. Furthermore, when the radius is $1000$, Figures~\ref{fig:col_fil_1000_gap}--\ref{fig:col_fil_1000_subopt} show that for all optimality measures, the log-adaptive step-size admits arbitrarily fast sublinear convergence rates and strictly outperforms the two fixed-$\ell$ step-sizes.

\begin{remark}[Open question: growth settings of collaborative filtering]
To the best of our knowledge, it is currently not known whether collaborative filtering can be captured by any of the growth settings used throughout this paper.  Thus the accelerated convergence rates in Figures~\ref{fig:col_fil_1000_gap}--\ref{fig:col_fil_1000_subopt} and~\ref{fig:col_fil_3000_subopt} warrant further investigation. We plan to address this question in future work. Some of the authors previously encountered a similar phenomenon in \cite{wirth2023acceleration}.
\end{remark}

\section{Discussion}\label{sec.discussion}
We demonstrated theoretically and via numerical experiments that the proposed log-adaptive step-size $\eta_t=\frac{2+\log(t+1)}{t+2+\log(t+1)}$ converges at least as fast as fixed-$\ell$ step-sizes of the form $\eta_t=\frac{\ell}{t+\ell}$ for any $\ell\in\N_{\ge 2}$, up to polylogarithmic factors. Furthermore, in the strong $(M,r)$-growth settings, the log-adaptive step-size admits arbitrarily fast sublinear convergence rates and strictly outperforms fixed-$\ell$ step-sizes. Furthermore, the proposed step-size does not require more overhead than fixed-$\ell$ step-sizes. Henceforth, the proposed step-size should therefore be considered the default open-loop step-size for \fw{}. To facilitate widespread adoption, we integrated adaptive step-sizes into the \texttt{FrankWolfe.jl} package.

\subsubsection*{Acknowledgements}
Research reported in this paper was partially supported by the Deutsche Forschungsgemeinschaft (DFG, German Research Foundation) under Germany's Excellence Strategy – The Berlin Mathematics Research Center MATH$^+$ (EXC-2046/1, project ID 390685689, BMS Stipend).
Research reported in this paper was also partially supported by the Bajaj Family Chair at the Tepper School of Business, Carnegie Mellon University.

\bibliographystyle{plain}

\bibliography{bibliography.bib}

\begin{thebibliography}{10}

\bibitem{alayrac2016unsupervised}
Jean-Baptiste Alayrac, Piotr Bojanowski, Nishant Agrawal, Josef Sivic, Ivan Laptev, and Simon Lacoste-Julien.
\newblock Unsupervised learning from narrated instruction videos.
\newblock In {\em Proceedings of the IEEE Conference on Computer Vision and Pattern Recognition}, pages 4575--4583, 2016.

\bibitem{allen2017linear}
Zeyuan Allen-Zhu, Elad Hazan, Wei Hu, and Yuanzhi Li.
\newblock Linear convergence of a {F}rank-{W}olfe type algorithm over trace-norm balls.
\newblock {\em Proceedings of Advances in Neural Information Processing Systems}, 30, 2017.

\bibitem{bach2013learning}
Francis Bach et~al.
\newblock Learning with submodular functions: A convex optimization perspective.
\newblock {\em Foundations and Trends{\textregistered} in Machine Learning}, 6(2-3):145--373, 2013.

\bibitem{bach2012equivalence}
Francis Bach, Simon Lacoste-Julien, and Guillaume Obozinski.
\newblock On the equivalence between herding and conditional gradient algorithms.
\newblock In {\em Proceedings of the International Conference on Machine Learning}, pages 1355--1362. PMLR, 2012.

\bibitem{baskaran2022distribution}
Vishal~Athreya Baskaran, Jolene Ranek, Siyuan Shan, Natalie Stanley, and Junier~B Oliva.
\newblock Distribution-based sketching of single-cell samples.
\newblock In {\em Proceedings of the 13th ACM International Conference on Bioinformatics, Computational Biology and Health Informatics}, pages 1--10, 2022.

\bibitem{besanccon2022frankwolfe}
Mathieu Besan{\c{c}}on, Alejandro Carderera, and Sebastian Pokutta.
\newblock Frankwolfe. jl: A high-performance and flexible toolbox for {F}rank--{W}olfe algorithms and conditional gradients.
\newblock {\em INFORMS Journal on Computing}, 34(5):2611--2620, 2022.

\bibitem{bian2017guaranteed}
Andrew~An Bian, Baharan Mirzasoleiman, Joachim Buhmann, and Andreas Krause.
\newblock Guaranteed non-convex optimization: Submodular maximization over continuous domains.
\newblock In {\em Proceedings of the International Conference on Artificial Intelligence and Statistics}, pages 111--120. PMLR, 2017.

\bibitem{bojanowski2015weakly}
Piotr Bojanowski, R{\'e}mi Lajugie, Edouard Grave, Francis Bach, Ivan Laptev, Jean Ponce, and Cordelia Schmid.
\newblock Weakly-supervised alignment of video with text.
\newblock In {\em Proceedings of the IEEE International Conference on Computer Vision}, pages 4462--4470, 2015.

\bibitem{braun2022conditional}
G{\'a}bor Braun, Alejandro Carderera, Cyrille~W Combettes, Hamed Hassani, Amin Karbasi, Aryan Mokhtari, and Sebastian Pokutta.
\newblock Conditional gradient methods.
\newblock {\em arXiv preprint arXiv:2211.14103}, 2022.

\bibitem{bugg2022nonnegative}
Caleb~Xavier Bugg, Chen Chen, and Anil Aswani.
\newblock Nonnegative tensor completion via integer optimization.
\newblock In {\em Proceedings of Advances in Neural Information Processing Systems}, 2022.

\bibitem{carderera2021simple}
Alejandro Carderera, Mathieu Besan{\c{c}}on, and Sebastian Pokutta.
\newblock Simple steps are all you need: {F}rank-{W}olfe and generalized self-concordant functions.
\newblock {\em Proceedings of Advances in Neural Information Processing Systems}, 34:5390--5401, 2021.

\bibitem{clarkson2012sublinear}
Kenneth~L Clarkson, Elad Hazan, and David~P Woodruff.
\newblock Sublinear optimization for machine learning.
\newblock {\em Journal of the ACM (JACM)}, 59(5):1--49, 2012.

\bibitem{combettes2021complexity}
Cyrille~W Combettes and Sebastian Pokutta.
\newblock Complexity of linear minimization and projection on some sets.
\newblock {\em Operations Research Letters}, 49(4):565--571, 2021.

\bibitem{dunn1978conditional}
Joseph~C Dunn and S~Harshbarger.
\newblock Conditional gradient algorithms with open loop step size rules.
\newblock {\em Journal of Mathematical Analysis and Applications}, 62(2):432--444, 1978.

\bibitem{frank1956algorithm}
Marguerite Frank and Philip Wolfe.
\newblock An algorithm for quadratic programming.
\newblock {\em Naval Research Logistics Quarterly}, 3(1-2):95--110, 1956.

\bibitem{freund2017extended}
Robert~M Freund, Paul Grigas, and Rahul Mazumder.
\newblock An extended {F}rank-{W}olfe method with “in-face” directions, and its application to low-rank matrix completion.
\newblock {\em SIAM Journal on Optimization}, 27(1):319--346, 2017.

\bibitem{garber2015faster}
Dan Garber and Elad Hazan.
\newblock Faster rates for the {F}rank-{W}olfe method over strongly-convex sets.
\newblock In {\em Proceedings of the International Conference on Machine Learning}. PMLR, 2015.

\bibitem{garber2016linear}
Dan Garber and Ofer Meshi.
\newblock Linear-memory and decomposition-invariant linearly convergent conditional gradient algorithm for structured polytopes.
\newblock In {\em Proceedings of the International Conference on Neural Information Processing Systems}, pages 1009--1017. PMLR, 2016.

\bibitem{garber2018fast}
Dan Garber, Shoham Sabach, and Atara Kaplan.
\newblock Fast generalized conditional gradient method with applications to matrix recovery problems.
\newblock {\em arXiv preprint arXiv:1802.05581}, 3, 2018.

\bibitem{ghadimi2019conditional}
Saeed Ghadimi.
\newblock Conditional gradient type methods for composite nonlinear and stochastic optimization.
\newblock {\em Mathematical Programming}, 173:431--464, 2019.

\bibitem{guo2017efficient}
Xiawei Guo, Quanming Yao, and James Kwok.
\newblock Efficient sparse low-rank tensor completion using the {F}rank-{W}olfe algorithm.
\newblock In {\em Proceedings of the AAAI Conference on Artificial Intelligence}, volume~31, 2017.

\bibitem{harper2015movielens}
F~Maxwell Harper and Joseph~A Konstan.
\newblock The movielens datasets: History and context.
\newblock {\em ACM Transactions on Interactive Intelligent Systems}, 5(4):1--19, 2015.

\bibitem{holloway1974extension}
Charles~A Holloway.
\newblock An extension of the {F}rank and {W}olfe method of feasible directions.
\newblock {\em Mathematical Programming}, 6(1):14--27, 1974.

\bibitem{huber1992robust}
Peter~J Huber.
\newblock Robust estimation of a location parameter.
\newblock {\em Breakthroughs in statistics: Methodology and distribution}, pages 492--518, 1992.

\bibitem{jaggi2013revisiting}
Martin Jaggi.
\newblock Revisiting {F}rank-{W}olfe: Projection-free sparse convex optimization.
\newblock In {\em Proceedings of the International Conference on Machine Learning}, pages 427--435. PMLR, 2013.

\bibitem{joulin2014efficient}
Armand Joulin, Kevin Tang, and Li~Fei-Fei.
\newblock Efficient image and video co-localization with {F}rank-{W}olfe algorithm.
\newblock In {\em Computer Vision -- ECCV 2014}, pages 253--268, Cham, 2014. Springer International Publishing.

\bibitem{kerdreux2021projection}
Thomas Kerdreux, Alexandre d’Aspremont, and Sebastian Pokutta.
\newblock Projection-free optimization on uniformly convex sets.
\newblock In {\em Proceedings of the International Conference on Artificial Intelligence and Statistics}, pages 19--27. PMLR, 2021.

\bibitem{lacoste2015global}
Simon Lacoste-Julien and Martin Jaggi.
\newblock On the global linear convergence of {F}rank-{W}olfe optimization variants.
\newblock In {\em Proceedings of Advances in Neural Information Processing Systems}, pages 496--504, 2015.

\bibitem{lacoste2013block}
Simon Lacoste-Julien, Martin Jaggi, Mark Schmidt, and Patrick Pletscher.
\newblock Block-coordinate {F}rank-{W}olfe optimization for structural svms.
\newblock In {\em Proceedings of the International Conference on Machine Learning}, pages 53--61. PMLR, 2013.

\bibitem{lacoste2015sequential}
Simon Lacoste-Julien, Fredrik Lindsten, and Francis Bach.
\newblock Sequential kernel herding: {F}rank-{W}olfe optimization for particle filtering.
\newblock In {\em Proceedings of the International Conference on Artificial Intelligence and Statistics}, pages 544--552. PMLR, 2015.

\bibitem{lan2013complexity}
Guanghui Lan.
\newblock The complexity of large-scale convex programming under a linear optimization oracle.
\newblock {\em arXiv preprint arXiv:1309.5550}, 2013.

\bibitem{levitin1966constrained}
Evgeny~S Levitin and Boris~T Polyak.
\newblock Constrained minimization methods.
\newblock {\em USSR Computational Mathematics and Mathematical Physics}, 6(5):1--50, 1966.

\bibitem{li2021momentum}
Bingcong Li, Mario Coutino, Georgios~B Giannakis, and Geert Leus.
\newblock A momentum-guided {F}rank-{W}olfe algorithm.
\newblock {\em IEEE Transactions on Signal Processing}, 69:3597--3611, 2021.

\bibitem{mehta2007robust}
Bhaskar Mehta, Thomas Hofmann, and Wolfgang Nejdl.
\newblock Robust collaborative filtering.
\newblock In {\em Proceedings of the ACM Conference on Recommender Systems}, pages 49--56, 2007.

\bibitem{miech2017learning}
Antoine Miech, Jean-Baptiste Alayrac, Piotr Bojanowski, Ivan Laptev, and Josef Sivic.
\newblock Learning from video and text via large-scale discriminative clustering.
\newblock In {\em Proceedings of the IEEE International Conference on Computer Vision}, pages 5257--5266, 2017.

\bibitem{mokhtari2018conditional}
Aryan Mokhtari, Hamed Hassani, and Amin Karbasi.
\newblock Conditional gradient method for stochastic submodular maximization: Closing the gap.
\newblock In {\em Proceedings of the International Conference on Artificial Intelligence and Statistics}, pages 1886--1895. PMLR, 2018.

\bibitem{mu2016scalable}
Cun Mu, Yuqian Zhang, John Wright, and Donald Goldfarb.
\newblock Scalable robust matrix recovery: {F}rank--{W}olfe meets proximal methods.
\newblock {\em SIAM Journal on Scientific Computing}, 38(5):A3291--A3317, 2016.

\bibitem{nesterov2018complexity}
Yu~Nesterov.
\newblock Complexity bounds for primal-dual methods minimizing the model of objective function.
\newblock {\em Mathematical Programming}, 171(1-2):311--330, 2018.

\bibitem{osokin2016minding}
Anton Osokin, Jean-Baptiste Alayrac, Isabella Lukasewitz, Puneet Dokania, and Simon Lacoste-Julien.
\newblock Minding the gaps for block {F}rank-{W}olfe optimization of structured svms.
\newblock In {\em Proceedings of the International Conference on Machine Learning}, pages 593--602. PMLR, 2016.

\bibitem{ouyang2010fast}
Hua Ouyang and Alexander Gray.
\newblock Fast stochastic frank-wolfe algorithms for nonlinear svms.
\newblock In {\em Proceedings of the 2010 SIAM International Conference on Data Mining}, pages 245--256. SIAM, 2010.

\bibitem{pedregosa2018step}
Fabian Pedregosa, Armin Askari, Geoffrey Negiar, and Martin Jaggi.
\newblock Step-size adaptivity in projection-free optimization.
\newblock {\em arXiv preprint arXiv:1806.05123}, 2018.

\bibitem{Pena23}
Javier~F Pena.
\newblock Affine invariant convergence rates of the conditional gradient method.
\newblock {\em SIAM Journal on Optimization}, 33(4):2654--2674, 2023.

\bibitem{peyre2017weakly}
Julia Peyre, Josef Sivic, Ivan Laptev, and Cordelia Schmid.
\newblock Weakly-supervised learning of visual relations.
\newblock In {\em Proceedings of the IEEE International Conference on Computer Vision}, pages 5179--5188, 2017.

\bibitem{pokutta2023frankwolfe}
Sebastian Pokutta.
\newblock {The Frank-Wolfe algorithm: a short introduction}.
\newblock {\em {Jahresbericht der Deutschen Mathematiker-Vereinigung}}, 126:3--35, 1 2024.

\bibitem{rezaei2017background}
Behnaz Rezaei and Sarah Ostadabbas.
\newblock Background subtraction via fast robust matrix completion.
\newblock In {\em Proceedings of the IEEE International Conference on Computer Vision Workshops}, pages 1871--1879, 2017.

\bibitem{seguin2016instance}
Guillaume Seguin, Piotr Bojanowski, R{\'e}mi Lajugie, and Ivan Laptev.
\newblock Instance-level video segmentation from object tracks.
\newblock In {\em Proceedings of the IEEE Conference on Computer Vision and Pattern Recognition}, pages 3678--3687, 2016.

\bibitem{shalev2011large}
Shai Shalev-Shwartz, Alon Gonen, and Ohad Shamir.
\newblock Large-scale convex minimization with a low-rank constraint.
\newblock {\em arXiv preprint arXiv:1106.1622}, 2011.

\bibitem{tsuji2022pairwise}
Kazuma~K Tsuji, Ken’ichiro Tanaka, and Sebastian Pokutta.
\newblock Pairwise conditional gradients without swap steps and sparser kernel herding.
\newblock In {\em Proceedings of the International Conference on Machine Learning}, pages 21864--21883. PMLR, 2022.

\bibitem{wirth2023approximate}
Elias Wirth, Hiroshi Kera, and Sebastian Pokutta.
\newblock Approximate vanishing ideal computations at scale.
\newblock In {\em Proceedings of the International Conference on Learning Representations}, 2023.

\bibitem{wirth2023acceleration}
Elias Wirth, Thomas Kerdreux, and Sebastian Pokutta.
\newblock Acceleration of {F}rank-{W}olfe algorithms with open-loop step-sizes.
\newblock In {\em Proceedings of the International Conference on Artificial Intelligence and Statistics}, pages 77--100. PMLR, 2023.

\bibitem{wirth2023accelerated}
Elias Wirth, Javier Pena, and Sebastian Pokutta.
\newblock Accelerated affine-invariant convergence rates of the {F}rank-{W}olfe algorithm with open-loop step-sizes.
\newblock {\em Mathematical Programming}, pages 1--45, 2025.

\bibitem{wirth2022conditional}
Elias Wirth and Sebastian Pokutta.
\newblock Conditional gradients for the approximately vanishing ideal.
\newblock In {\em Proceedings of the International Conference on Artificial Intelligence and Statistics}, pages 2191--2209. PMLR, 2022.

\bibitem{wolfe1970convergence}
Philip Wolfe.
\newblock Convergence theory in nonlinear programming.
\newblock {\em Integer and Nonlinear Programming}, pages 1--36, 1970.

\end{thebibliography}

\end{document}